\documentclass[11pt]{article}

\usepackage{amstext,amssymb,amsmath,amsbsy}
\usepackage{hyperref}
\usepackage{amscd}
\usepackage{amsfonts}
\usepackage{indentfirst}
\usepackage{verbatim}
\usepackage{amsmath}
\usepackage{amsthm}
\usepackage{enumerate}
\usepackage{graphicx}
\usepackage{color, soul}
\usepackage[OT1]{fontenc}
\usepackage[latin1]{inputenc}
\usepackage[english]{babel}
\usepackage{amssymb}
\usepackage{subfig}
\usepackage{algorithm}
\usepackage{algorithmic}

\usepackage{mathtools}
\DeclarePairedDelimiter\ceil{\lceil}{\rceil}

\newcommand{\R}{\mathbb{R}}
\newcommand{\N}{\mathbb{N}}

\newcommand{\x}{{\bf x}}
\newcommand{\y}{{\bf y}}
\newcommand{\p}{{\bf p}}

\newcommand{\Div}{{\rm div}}
\newcommand{\ii}{{\rm i}}
\newcommand{\ik}{\ii k}

\newtheorem{Theorem}{Theorem}[section]
\newtheorem{Lemma}{Lemma}[section]

\newtheorem{Proposition}{Proposition}[section]

\newtheorem{remark}{Remark}[section]

\newtheorem*{Assumption*}{Assumption}

\newtheorem{problem}{Problem}[section]
\newtheorem*{problem*}{Problem}
\numberwithin{equation}{section}

\begin{document}

\title{A Carleman-based numerical method for quasilinear elliptic equations with over-determined boundary data and applications}

\author{Thuy T. Le\thanks{Department of Mathematics and Statistics, University of North Carolina at
Charlotte, Charlotte, NC, 28223, USA, \texttt{tle55@uncc.edu}.} \and Loc H. Nguyen\thanks{Department of Mathematics and Statistics, University of North Carolina at
Charlotte, Charlotte, NC, 28223, USA, \texttt{loc.nguyen@uncc.edu}.} \and Hung V. Tran\thanks{Department of Mathematics, University of Wisconsin Madison, Madison, WI, 53706, USA, \texttt{hung@math.wisc.edu} (corresponding
author).}} 


\date{}
\maketitle
\begin{abstract}
	We propose a new iterative scheme to compute the numerical solution to an over-determined boundary value problem for a general quasilinear elliptic PDE.
	The main idea is to  repeatedly solve its linearization by using the quasi-reversibility method with  a suitable Carleman weight function.
	The presence of the Carleman weight function allows us to employ  a Carleman estimate  to prove the convergence of the sequence generated by the iterative scheme above to the desired solution.
	The convergence of the iteration is fast at an exponential  rate without the need of an initial good guess.
	We apply this method to compute solutions to some general quasilinear elliptic equations and a large class of first-order Hamilton-Jacobi equations.
	Numerical results are presented.
\end{abstract}



\section{Introduction} \label{sec intr}

The main aim of this paper is to develop a numerical method based on Carleman estimates to solve quasilinear elliptic PDEs with over-determined boundary data.
We consider this new method as the second generation of Carleman-based numerical methods while the first generation is called the convexification, which will be mentioned in detail later.
Let $\Omega$ be an open and bounded domain in $\R^d$, $d \geq 2$, with smooth boundary $\partial \Omega$. 
Let $f$ and $g$ be two smooth functions defined on $\partial \Omega$. 
Let $F: \overline \Omega \times \R \times \R^d \to \R$ be a function in the class $C^2.$
Let $A=(a_{ij})_{i,j=1}^d: \overline \Omega \to \R^{d\times d}$ be $C^2$, symmetric, and positive definite, that is,
\[
\gamma |\xi|^2 \leq a_{ij}(\x) \xi_i \xi_j \leq \gamma^{-1} |\xi|^2 \quad \mbox{ for all } \x \in \overline\Omega, \ \xi \in \R^d,
\]
for some fixed $\gamma \in (0,1)$.
The following problem is of our interests. 

\begin{problem}
Assume that the over-determined boundary value problem
\begin{equation}
\left\{
	\begin{array}{ll}
		 -\Div (A(\x)\nabla u(\x))  + F(\x, u(\x), \nabla u(\x)) = 0 &\x \in \Omega,\\
		u(\x) = f(\x) &\x \in \partial\Omega,\\
		\partial_\nu u(\x) = g(\x) &\x \in \partial \Omega
	\end{array}
\right.
\label{main eqn}
\end{equation}
has a solution $u^*$ in $C^2(\overline \Omega)$. 
Compute the function $u^*$.
\label{p1}
\end{problem}

Problem \ref{p1} is motivated by a class of nonlinear inverse problems in PDEs, in which $f$ and $g$ are the data that can be measured.
One important goal of   inverse problems is to reconstruct the internal structure of a domain from boundary measurements,
which allow us to impose both Dirichlet and Neumann data of the unknown in \eqref{main eqn}.
Recently,  a unified framework to solve such inverse problems was developed by the research group of the first and second authors, which has two main steps.
In the first step, by introducing a change of variables, one derives a PDE of the form  \eqref{main eqn} from the given inverse problem, in which $f$ and $g$ can be computed directly  by the given boundary data.
In the second step, one numerically solves \eqref{main eqn} to find $u^*$. 
The knowledge of $u^*
$ directly yields that of the solution to the corresponding inverse problem under consideration.
See \cite{KhoaKlibanovLoc:SIAMImaging2020, LeNguyen:jiip2022} and the references therein for some works in this framework.
Moreover, this unified framework was successfully tested with experimental data in \cite{VoKlibanovNguyen:IP2020, Khoaelal:IPSE2021, KlibanovLeNguyenIPI2021}.
Another motivation to study Problem \ref{p1} is  to seek solutions to  Hamilton-Jacobi equations under the circumstance that the Neumann data of the unknown  can be computed by its Dirichlet data and the given form of the Hamiltonian, see e.g., \cite[Assumption 1.1 and Remark 1.1]{KlibanovNguyenTran:JCP2022}.
Since inverse problems are out of the scope of this paper, we only focus on the applications in solving quasilinear elliptic PDEs and first-order Hamilton-Jacobi equations.

A natural approach to solve \eqref{main eqn} is based on optimization. 
That means one sets the computed solution to \eqref{main eqn} as a minimizer of a mismatch functional, e.g.,
\begin{equation*}
	v \mapsto J(v) := \int_{\Omega} \big| -\Div(A(\x)\nabla v(\x)) + F(\x, v(\x), \nabla v(\x))\big|^2\, d\x 
	+ \mbox{a regularization term}
\end{equation*}
subject to the Cauchy boundary conditions $v|_{\partial \Omega} = f$ and $\partial_{\nu} v|_{\partial \Omega} = g.$
The methods based on optimization are widely used in the scientific community, especially in computational mathematics, physics and engineering.
Although effective and popular, the optimization-based approaches have some drawbacks:
\begin{enumerate}
	\item \label{drb1} In general, it is not clear that the obtained minimizer approximates the true solution to \eqref{main eqn}.
	\item \label{drb2} The mismatch functional $J$ is not convex, and it might have multiple minima and ravines (see an example in \cite{ScalesSmithFischerLjcp1992} for illustration). 
	To deliver reliable numerical solutions, one must know some good initial guesses of the true solutions.
	\item \label{drb3} The computation is expensive and time consuming.  
\end{enumerate}
Drawbacks \# \ref{drb1} and \#\ref{drb2}  can be treated by the convexification method, which is designed to globalize the optimization methods. 
The main idea of the convexification method is to employ some suitable Carleman weight functions to convexify  the mismatch functionals.
 The convexity of weighted mismatch functionals is rigorously proved by  Carleman estimates.
Several versions of the convexification method  have been developed in  \cite{KlibanovNik:ra2017, KhoaKlibanovLoc:SIAMImaging2020, Klibanov:sjma1997, Klibanov:nw1997, Klibanov:ip2015, KlibanovLeNguyenIPI2021} since it was first introduced in \cite{KlibanovIoussoupova:SMA1995}. 
Moreover, we recently discovered that the convexification method can be used to solve a large class of first-order Hamilton-Jacobi equations \cite{KlibanovNguyenTran:JCP2022}.

In this paper, we introduce a new method to solve \eqref{main eqn} based on linearization and Carleman estimates.
Like the convexification method, our method delivers a reliable solution to \eqref{main eqn} without requiring a good initial guess. This fact is rigorously proved.
Unlike the convexification method which is time consuming, our new method quickly provides the desired solutions. 
Its converge rate is $O(\theta^n)$ as $n \to \infty$ for some $\theta \in (0, 1).$

We find the numerical solution to \eqref{main eqn} by repeatedly solving the linearization of \eqref{main eqn} by  a new ``Carleman weighted" quasi-reversibility method.
The classical quasi-reversibility method was first proposed in \cite{LattesLions:e1969}, and it has been studied intensively since then (see \cite{Klibanov:anm2015} for a survey).
By a Carleman weighted quasi-reversibility method, we mean that we let  a suitable Carleman weight function involve in the cost functional suggested by the classical quasi-reversibility method.
The presence of the Carleman weight function is the key for us to prove our convergence theorem.
Our process to solve Problem \ref{p1} is as follows.
We first choose any initial solution that might be far away from the true one. 
Denote this initial solution by the function $u_0$.
Linearizing \eqref{main eqn} about $u_0$, we obtain a linear PDE. 
We then solve this linear PDE by the Carleman weighted quasi-reversibility method to obtain an updated solution $u_1$.
Using the Carleman weighted quasi-reversibility method rather than the classical one in this step is the key to our success.
By iteration, we repeat this step to construct a sequence $\{u_n\}_{n \geq 0}$.
The convergence of this sequence to the true solution to \eqref{main eqn} is proved by using Carleman estimates.
We then apply this method to numerically solve some quasilinear elliptic equations.
It is important to note that our approach works well for systems of quasilinear elliptic PDEs too.

\begin{remark}
In general, \eqref{main eqn} is over-determined and it might have no solution, especially when the boundary data contains some noise.
Our iteration and linearization method using the Carleman weighted quasi-reversibility method in each step still delivers a function that ``most fits" \eqref{main eqn}. 

On the other hand, \eqref{main eqn} with only Dirichlet boundary condition, that is, the equation
\begin{equation}\label{eq:eigen}
\left\{
	\begin{array}{ll}
		 -\Div (A(\x)\nabla u(\x))  + F(\x, u(\x), \nabla u(\x)) = 0 &\x \in \Omega,\\
		u(\x) = f(\x) &\x \in \partial\Omega
	\end{array}
\right.
\end{equation}
might have many solutions since we do not impose structural conditions on $F$.
For example, in case $F(\x,u(\x),\nabla u(\x))= \lambda u$ for some $\lambda \in \R$ and $f=0$, \eqref{eq:eigen} becomes an eigenvalue problem with possibly many solutions as eigenfunctions.
In such cases, requiring the additional Neumann boundary condition is then natural, and \eqref{main eqn} is not over-determined.
\end{remark}

Next, we use our  method to solve some first-order Hamilton-Jacobi equations.
More precisely, to find viscosity solutions to the first-order equation 
\[
F(\x, u(\x), \nabla u(\x)) = 0 \quad \x \in \Omega,
\]
we use the vanishing viscosity procedure and  consider, for $0 < \epsilon_0 \ll1$,
\[
	 -\epsilon_0 \Delta u(\x)  + F(\x, u(\x), \nabla u(\x)) = 0 \quad \x \in \Omega,\\
\]
with given Cauchy boundary data.
Our new method is robust in the sense that it works for general nonlinearity $F(\x, u, \nabla u)$ that might not be convex in $\nabla u.$
We  refer the readers to \cite{CrandallLions83, CrandallEvansLions84, Tran19} and the references therein for the theory of viscosity solutions.
A weakness of our new approach in computing the viscosity solutions to Hamilton-Jacobi equations is that we need to require both Dirichlet and Neumann data of the unknown $u$. 
See \cite[Remark 1.1]{KlibanovNguyenTran:JCP2022} for some circumstances that this requirement is fulfilled. 
There have been many important methods to solve Hamilton-Jacobi equations in the literature.
For finite difference monotone and consistent schemes of first-order equations and applications, see \cite{BS-num, CL-rate,  OsFe, Sethian,  Sou1} for details and recent developments.
If $F=F(\x,u,\nabla u)$ is convex in $\nabla u$ and satisfies appropriate conditions, it is possible to construct some semi-Lagrangian approximations by the discretization of the Dynamical Programming Principle associated to the problem, see \cite{FaFe1, FaFe2} and the references therein.

The paper is organized as follows.
In Section \ref{sec 2222}, we recall a Carleman estimate and two examples about inverse problems in which Problem \ref{p1} appears.
In Section \ref{sec 2},  we introduce the new iterative method based on linearization and Carleman estimates.
In Section \ref{sec conv}, we prove the convergence of our method.
Some numerical results for quasilinear elliptic equations and first-order Hamilton-Jacobi equations are presented in Section \ref{sec num}.
Concluding remarks are given in Section \ref{sec 5}.

\section{Preliminaries} \label{sec 2222}

In this section, we recall a  Carleman estimate,  which plays a key role for the proof of the convergence theorem in this paper. 
We then  present an  inverse scattering problem in which Problem \ref{main eqn} appears.
\subsection{A Carleman estimate}

In this section, we present a simple form of Carleman estimates. 
Carleman estimates were first employed to prove the unique continuation principle, see e.g., \cite{Carleman:1933, Protter:1960AMS},
 and they quickly became a powerful tool in many areas of PDEs afterwards.
Let $\x_0$ be a point in $\R^d \setminus \overline \Omega$ such that $r(\x) = |\x - \x_0| > 1$ for all $\x \in \Omega.$
For each $\beta > 0$, define 
\begin{equation}
	\mu_\beta(\x) = r^{-\beta}(\x)= |\x - \x_0|^{-\beta}  \quad \mbox{for all } \x \in \overline \Omega.
	\label{mu}
\end{equation}

We have the following lemma.

\begin{Lemma}[Carleman estimate]\label{lem:Carleman}
There exist  positive constants $\beta_0,\lambda_0$ depending only on $\x_0$, $\Omega$, $\gamma$,  and $d$ such that for all function $v \in C^2(\overline \Omega)$ satisfying
	\begin{equation}
		v(\x) = \partial_{\nu} v(\x) = 0 \quad \mbox{for all } 
		\x \in \partial \Omega,
		\label{3.1}
	\end{equation}
	the following estimate holds true
	\begin{equation}
		\int_{\Omega} e^{2\lambda \mu_\beta(\x)}|\Div(A \nabla v) |^2\,d\x
		\geq
		 C\lambda  \int_{\Omega}  e^{2\lambda \mu_\beta(\x)}|\nabla v(\x)|^2\,d\x
		+ C\lambda^3  \int_{\Omega}   e^{2\lambda \mu_\beta(\x)}|v(\x)|^2\,d\x	
		\label{Car est}	
	\end{equation}
	for all $\beta \geq \beta_0$ and $\lambda \geq \lambda_0$. 
	Here, $C = C( \x_0,  \Omega, \gamma, d, \beta) > 0$ depends only on the listed parameters.
	\label{lemma carl}
\end{Lemma}

\begin{proof}
Lemma \ref{lemma carl} is a direct consequence of \cite[ Lemma 5]{MinhLoc:tams2015}.
Let $0< R_1 < 1$ and $R_3 \gg 1$ such that $\Omega \Subset B(\x_0, R_3) \setminus \overline{B(\x_0, R_1)}$.
Here, $B(\x_0, s) = \{\y \in \R^d: |\y - \x_0| < s\}$ for $s > 0.$
Extend $v$ to the whole $\R^d$ such that $v(\x) = 0$ for all $\x \in \R^d \setminus \Omega$. 
Using a change of variable $\x \mapsto \x - \x_0$ and \cite[ Lemma 5]{MinhLoc:tams2015}, there exists a number $\beta_0 \geq 1$ depending on $\gamma$, $R_1$ and $R_3$ such that for all $\beta \geq \beta_0$ and $|\lambda| \geq 2R_3^{-\beta}$, we have
\begin{multline}
	\int_{B(\x_0, R_3) \setminus \overline{B(\x_0, R_1)}} r(\x)^{\beta + 2} e^{2\lambda r(\x)^{-\beta}} 
	|\Div(A\nabla v)|^2 \,d\x
	\\
	\geq 
	C \int_{B(\x_0, R_3) \setminus \overline{B(\x_0, R_1)}}  e^{2\lambda r^{-\beta}(\x)} |\lambda| \beta\big(\beta^3 \lambda^2 r^{-2\beta - 2}(\x) |v|^2 + |\nabla v|^2\big)\,d\x
	\\
	- C\int_{\partial (B(\x_0, R_3) \setminus \overline{B(\x_0, R_1)})} |\lambda| \beta e^{2\lambda r^{-\beta}(\x)} (r(\x) |\nabla v(\x)|^2 
	\\
	+ \beta^2 \lambda^2 r^{-2\beta - 1}(\x)|v(\x)|^2)\, d\sigma(\x)
	\label{Carl Minh}
\end{multline}
for some constant $C$ depending only on $d$ and $\gamma$.
Since $v = 0$ on $\big(B(\x_0, R_3) \setminus \overline{B(\x_0, R_1)}\big) \setminus \Omega$ and since $\Omega \Subset B(\x_0, R_3) \setminus \overline{B(\x_0, R_1)}$, 
allowing $C$ to depend on $\beta$ and $\Omega$, we deduce the Carleman estimate \eqref{Car est} from \eqref{Carl Minh}.
\end{proof}

An alternative way to obtain \eqref{Car est} is to apply the Carleman estimate in  \cite[Chapter 4, \S 1, Lemma 3]{Lavrentiev:AMS1986} for general parabolic operators. 
The arguments to obtain \eqref{Car est} using \cite[Chapter 4, \S 1, Lemma 3]{Lavrentiev:AMS1986} are similar to that in \cite[Section 3]{LeNguyenNguyenPowell:JOSC2021} with the Laplacian replaced by the operator $\Div (A\nabla \cdot)$. 

\begin{remark}
	We specially draw the reader's attention to different forms of Carleman estimates for all three main kinds of differential operators (elliptic, parabolic and hyperbolic) and their applications in inverse problems and computational mathematics \cite{BeilinaKlibanovBook, BukhgeimKlibanov:smd1981, KlibanovLiBook}. 
	It is worth mentioning that some Carleman estimates hold true for all functions $v$ satisfying $v|_{\partial \Omega} = 0$ and $\partial_{\nu} v|_{\Gamma} = 0$ where $\Gamma$ is a part of $\partial \Omega$, see e.g., \cite{KlibanovNguyenTran:JCP2022, NguyenLiKlibanov:2019}, which can be used to solve quasilinear elliptic PDEs with the boundary data partly given.
\end{remark}
	
\subsection{An  inverse scattering problem} \label{inverse1}

As mentioned in Section \ref{sec intr}, Problem \ref{p1} arises from nonlinear inverse problems. 
We present here an important example in the context of inverse scattering problems in the frequency domain.
Let $c: \R^d \to [1, \infty)$ be the spatially distributed dielectric constant of the medium and $[\underline k, \overline k]$ be an interval of wavenumbers with $\underline k >0$. 
Since the dielectric constant of the air is 1, we set $c(\x) = 1$ for all $\x \in \R^d \setminus \Omega.$
For each $k \in [\underline k, \overline k],$ let $w(\x, k)$, $\x \in \R^d$, be the wave field generated by a point source at $\x_0 \in \R^d \setminus \Omega$ with wavenumber $k$.
The function $w$ is governed by the Helmholtz equation and the Sommerfeld radiation condition
\begin{equation}
	\left\{
		\begin{array}{ll}
			\Delta w(\x, k) + k^2 c(\x)w(\x, k) = -\delta(\x - \x_0)  &\x \in \R^d,\\
			(\frac{\partial}{\partial {|\x|}}  - \ik) w(\x, k) = o(|\x|^{(1 - d)/2}) &|\x| \to \infty.
		\end{array}
	\right.
\end{equation}

The inverse scattering problem is formulated as follows.
\begin{problem}[Inverse Scattering Problem]
Compute the function $c(\x)$, $\x \in \Omega$, from the measurements of 
\begin{equation}
	f_1(\x, k) = w(\x, k) \quad
	\mbox{and}
	\quad
	f_2(\x, k) = \partial_{\nu}w(\x, k) 
\end{equation}
for all $\x \in \partial \Omega$, $k \in [\underline k, \overline k].$
\end{problem}
The knowledge of the function $c$ partly provides the internal structure of the domain $\Omega$.
In other words, solving the inverse scattering problem allows us to examine a domain from external measurements, which has applications in security, sonar imaging,  geographical exploration, medical imaging, near-field optical microscopy, nano-optics, see, e.g., \cite{ColtonKress:2013} and references therein for more details.
%
%
%
There have been many important methods to solve inverse scattering problems in the literature. 
Each method has its own advantages and disadvantages. 
A common drawback of the widely-used method based on optimization to solve  inverse scattering problems is the need of a good initial guess of the true solution $c$. 
We recall from \cite{LeNguyen:JSC2022} a method to solve the above inverse scattering problem  in which such a need is relaxed.
Denote by
\[
z(\x,k)=\frac{1}{k^2}\log\Big(\frac{w(\x, k)}{w_0(\x, k)}\Big) \quad \text{ for all } \x \in \Omega, k \in  [\underline k, \overline k],
\]
where $w_0(\x, k) = \frac{e^{\ik |\x - \x_0|}}{4\pi |\x - \x_0|}.$
Then, $z$ satisfies
\[
\Delta z(\x,k) + k^2 |\nabla z(\x,k)|^2 + \frac{2 \nabla z(\x,k) \cdot \nabla w_0(\x,k)}{w_0(\x,k)} = - c(\x)+1
\]
for all $\x \in \Omega, k \in  [\underline k, \overline k].$
Let $\{\Psi_n\}_{n \geq 1}$ be the orthonormal basis of $L^2(\underline k, \overline k)$ introduced in \cite{Klibanov:jiip2017} 
and define
\begin{equation}
	z_n(\x) = \int_{\underline k}^{\overline k} z(\x,k) \Psi_n(k) \,dk
	\quad \text{ for } n \geq 1, \x \in \Omega.
	\label{zn}
\end{equation}
 We approximate
\[
v(\x,k) =\sum_{i=1}^\infty z_i(\x) \Psi_i(k) \approx \sum_{i=1}^N z_i(\x) \Psi_i(k),
\]
for a suitable cut-off number $N\in \N$.
Then,
the vector $Z_N = (z_1, z_2, \dots, z_N)$ ``approximately" satisfies the system
\begin{equation}
\sum_{i = 1}^N s_{li} \Delta z_i(\mathbf{x}) + \sum_{i, j = 1}^N
a_{lij}\nabla z_i(\mathbf{x}) \cdot \nabla z_j(\mathbf{x}) + \sum_{i = 1}^N
B_{li}(\mathbf{x}) \cdot \nabla z_i(\mathbf{x}) = 0  \label{6.11}
\end{equation}
where 
\begin{equation*}
\left\{  
\begin{array}{l}
s_{li} = \displaystyle\int_{\underline k}^{\overline k} \Psi_i^{\prime }(k)
\Psi_l(k) \,dk, \\ 
a_{lij} = \displaystyle 2 \int_{\underline k}^{\overline k} \Big(k^2
\Psi_i(k)\Psi_j^{\prime }(k) + k \Psi_i(k)\Psi_j(k)\Big)\Psi_l(k)\,dk, \\ 
B_{li}(\mathbf{x}) = \displaystyle 2 \int_{\underline k}^{\overline k}\Big( %
\Psi_i^{\prime }(k)\frac{\nabla w_0(\mathbf{x}, k)}{w_0(\mathbf{x}, k)} +
\Psi_i(k) \partial_k \frac{\nabla w_0(\mathbf{x}, k)}{w_0(\mathbf{x}, k)} %
\Big)\Psi_l(k)\,dk%
\end{array}
\right. 
\end{equation*}
for all $i, j, l \in \{1, \dots, N\}$ and $\mathbf{x} \in \Omega$,
see \cite[Section 6]{LeNguyen:JSC2022} for details.
The Dirichlet and Neumann boundary conditions for $Z_N$ can be computed by the knowledges of $f_1$, $f_2$, and \eqref{zn}.
Solving the system of quasilinear elliptic equations \eqref{6.11} with the provided Dirichlet and Neumann data is basically a goal of Problem \ref{p1}. 
Doing so is the key step to compute $c$. 
See \cite{LeNguyen:JSC2022} for convexification method to compute $Z_N$ and the procedure to obtain $c$ from the knowledge of $Z_N.$


\subsection{Electrical impedance tomography}

We present here another potential application of our study in this paper to the 3D electrical impedance tomography (EIT) problem,  so-called the 3D Calder\'on problem, with only a part of the Dirichlet to Neumann map is given.
Let $\Omega = (-R, R)^3$ for some $R > 0$. 
Let the line of source be defined as
\begin{equation*} 
	L_{\rm sc} = \{\x_\alpha = (\alpha, 0, -R):   |\alpha| \leq R\} \subset \partial \Omega.
\end{equation*}
For each $\alpha \in [-R, R]$, let $\x_\alpha = (\alpha, 0, -R) \in L_{\rm sc}$ and $w = w(\x, \x_\alpha)$, $\x \in \Omega$, be the solution to
\begin{equation}
	\left\{
		\begin{array}{rcll}
			\Div\big(a(\x)\nabla w(\x, \x_\alpha)\big) &=& 0 &\x \in \Omega,\\
			w(\x, \x_\alpha) &=& f_1(\x, \x_\alpha) &\x \in \partial \Omega.
		\end{array}
	\right.
	\label{8}
\end{equation}
Here, $f_1(\x, \x_\alpha) > 0$ is a smooth approximation of the Dirac delta $\delta_0(\x - \x_\alpha)$. 
In EIT, $f_1(\x, \x_\alpha)$ represents the boundary electric voltage.
The EIT problem is formulated as follows.
\begin{problem}[Electrical impedance tomography]
	Determine the electric conductivity $a(\x)$, $\x \in \Omega,$ from the boundary measurement of the electric current
	$		
		f_2(\x, \x_\alpha) = \partial_\nu w(\x, \x_\alpha)
	$
	for all $\x \in \partial \Omega,$ $\x_\alpha \in L_{\rm sc}.$
 	\label{ip 11}
 \end{problem}
 
 \begin{remark}
 	Problem \ref{ip 11} only requests the data generated by the source moving on $L_{\rm sc}$. The dimension of this set of data is 3, including 1 dimension of the orbit $L_{\rm sc}$ of the source and 2 dimensions of the measurement surface $\partial \Omega$. This feature makes Problem \ref{ip 11} not over-determined because our goal is to reconstruct a 3D function $a$.
	This is unlike most of the works studying the EIT problem that request the whole Dirichlet to Neumann map $\Gamma: H^{1/2}(\partial \Omega) \to H^{-1/2}(\partial \Omega)$. Since $H^{1/2}(\partial \Omega)$, the domain of $\Gamma$, has uncountably infinity dimensions, the data requested by the Dirichlet to Neumann map is highly over-determined.
 \end{remark}
 

%

The EIT problem arises from bio-medical imaging; especially, in detecting early cancerous tumors in living tissues without operation. 
Most publications for the EIT problem studied the question how to reconstruct the electric conductivity $a(\x)$, $\x \in \Omega$ for some domain $\Omega$, from the Dirichlet to Neumann or the Neumann to Dirichlet map. 
We refer the reader to  \cite{Calderon:1980, Kohn:cpam1985, Nachman:am1996, SylvesterUhlmann:am1987} for some well-known uniqueness results of the EIT problem.
We list here a few publications for some  effective approaches to numerically solve the EIT problem: the D-bar method \cite{Knudsen:ip2009, Mueller:2020, Murphy:2009}, the methods based on optimization \cite{Gonzalez:camwa2017, Zhou:pm2015} and the convexification method \cite{KlibanovLiZhang:ip2019}. 
%
Although effective, these methods have drawbacks.
Firstly, the D-bar method is designed only for 2D.
Secondly, the methods based on optimization request good initial guess of the true solution.
And finally, the convexification method is time consuming.
Naturally, a new computational method should be studied in this direction. 
Our potential method to solve the 3D EIT problem is to reduce this inverse problem to a problem of the form \eqref{main eqn}.
Introduce the change of variable  
\begin{equation}
z(\x, \x_\alpha) = \log \big(\sqrt{a(\x)} w(\x, \x_\alpha)\big)
\quad \mbox{for all } \x \in \Omega, \x_\alpha \in L_{\rm sc}.
\label{zlog}
\end{equation}
%
 Let $\{\Psi_{n}\}_{n = 1}^\infty$ be the orthonormal basis of $L^2(-R, R)$ first introduced in \cite{Klibanov:jiip2017}. 
Like in Section \ref{inverse1}, we approximate the function $z(\x, \x_\alpha)$ 
as
\begin{equation}
	z(\x, \x_\alpha) = \sum_{n = 1}^\infty z_n(\x) \Psi_n(\alpha) \simeq \sum_{n = 1}^N z_n(\x) \Psi_n(\alpha)
	\quad 
	\label{3..3}
\end{equation}
for all $\x \in \Omega, \x_\alpha \in L_{\rm sc}$,
for some cut-off number $N$ where
\begin{equation}
	z_n(\x) = \int_{-R}^{R}z(\x, \x_\alpha)\Psi_n(\alpha) d\alpha 
\label{zn1}
\end{equation}
 for all $n \geq 1.$
 We choose $N$ such that the approximation in \eqref{3..3}  ``numerically" holds  for all $\x \in \partial \Omega$ where the data are known. 
 Then, it is not hard to verify, see \cite{KlibanovLiZhang:ip2019}, that 
 the vector $Z_N=(z_1,\ldots,z_N)$ satisfies the system
 \begin{equation}
 	\sum_{n = 1}^N s_{mn} \Delta z_n(\x) +  \sum_{n, l = 1}^N a_{mnl}\nabla z_n(\x) \cdot \nabla z_l(\x) = 0
	\quad \mbox{for all } \x \in \Omega, 
	\label{16}
 \end{equation}
 for each $m \in \{1, \dots, N\}$,
 where
 \begin{align*}
 	s_{mn} &= \int_{-R}^{R} \Psi_n'(\alpha) \Psi_m(\alpha)d\alpha,
	\\
	a_{mnl} &=2\int_{-R}^{R} \Psi_n(\alpha)\Psi_l'(\alpha) \Psi_m(\alpha)d\alpha.
 \end{align*}
 The Dirichlet and Neumann boundary conditions for $Z_N$ can be computed by the knowledges of the Dirichlet condition $f_1$, the given Neumann data $f_2$, \eqref{zlog} and \eqref{zn1}.
Solving the system of quasilinear elliptic equations \eqref{16} with the provided Dirichlet and Neumann data is basically a goal of Problem \ref{p1}.
We refer the reader to \cite{KlibanovLiZhang:ip2019} the step of computing $a$ from the knowledge of the vector $Z_N$.

The discussion above for the inverse scattering problem and the EIT problem motivates us to study Problem \ref{p1}. 
Since these inverse problems are out of the scope of this paper, we only mention them here to explain the significance of Problem \ref{p1}. 
In future works, we will use our solver for Problem \ref{p1} to solve these two important inverse problems.
 

\section{The iteration and linearization approach for Problem \ref{p1}} \label{sec 2}

Our approach to solve \eqref{main eqn} is based on linearization and iteration.
Assume that the solution $u^*$ to \eqref{main eqn} is in the space $H^p(\Omega)$ for some $p > \ceil{d/2} + 2$ where $\ceil{d/2}$ is the smallest integer that is greater than $d/2.$
Define the set of admissible solutions
\begin{equation}
	W = \big\{
		\varphi \in H^p(\Omega)\;:\, \varphi|_{\partial \Omega} = f, \partial_{\nu} \varphi|_{\partial \Omega} = g
	\big\}.
\end{equation}
Then, the assumption in Problem \ref{p1} implies that $W \ne \emptyset$, and $u^* \in W$.
We now construct a sequence $\{u_{n}\}_{n \geq 0}$ that converges to the solution $u^*$ to \eqref{main eqn}.
Take a function $u_0 \in W.$ 
Assume by induction that we have the knowledge of $u_{n}$ for some $n \geq 0$.
We find $u_{n+1}$ as follows.
Assume that $u^* = u_n + h$ for some $h \in H_0$ where
\[
	H_0 = \{\varphi \in H^p(\Omega)\,:\, \varphi|_{\partial \Omega} = 0, \partial_{\nu} \varphi|_{\partial \Omega} = 0\}.
\]
Plugging $u^* = u_n + h$ into \eqref{main eqn}, we have
\begin{equation}
	-\Div(A(\x) \nabla(u_n + h)(\x)) + F(\x, u_n(\x) + h(\x), \nabla u_n(\x) + \nabla h(\x) ) = 0
	\label{2.2}
\end{equation}
for all $\x \in \Omega.$
Heuristically, we assume at this moment that $h$ is small, that is, $\|h\|_{C^1(\Omega)} \ll 1$. 
\begin{remark}
The temporary assumption $\|h\|_{C^1(\Omega)} \ll 1$ is imposed for the suggestion in establishing a numerical scheme to find $u_{n + 1}$ while this condition is completely relaxed  in the proof of the convergence theorem.
\end{remark} 
By Taylor's expansion, we approximate \eqref{2.2} as
\begin{equation}
	-\Div(A \nabla u_n) -\Div(A \nabla h)+ F(\x, u_n(\x), \nabla u_n(\x)) + D\mathcal F(u_n) h = 0
	\label{2.3}
\end{equation}
for all $\x \in \Omega$
where
\[
	D\mathcal F(v)h = F_s(\x, v(\x), \nabla v(\x)) h + \nabla_{\p} F(\x, v(\x), \nabla v(\x))\cdot \nabla h(\x)
\] 
for all $v \in H^p(\Omega)$.
Here, $F_s$ and $\nabla_{\p} F$ are the partial derivative of $F$ with respect to its second variable and its gradient vector with respect to the third variable, respectively.

The next step is to compute a function $h \in H_0$ satisfying \eqref{2.3}. 
Since there is no guarantee for the existence of such a function $h$, we only compute a ``best fit" function $h$ by the Carleman-based quasi-reversibility method described below. 
For each $v \in H^p(\Omega)$, $\lambda > 1$, $\beta > 0$ and $\eta > 0$, define the functional $J_v^{\lambda, \beta, \eta}: H_0 \to \R$ as
\begin{multline*}
	J_v^{\lambda, \beta, \eta}(\varphi) = \int_{\Omega} e^{2\lambda \mu_\beta(\x)}\Big|
		-\Div(A \nabla  \varphi )- \Div(A \nabla v) + F(\x, v(\x), \nabla v(\x))  
		\\
		+ D\mathcal F(v) \varphi
	\Big|^2\,d\x
	+ \eta \|v + \varphi\|^2_{H^p(\Omega)}.
\end{multline*}
Here, the function $\mu_\beta$ is defined in \eqref{mu} and $\eta\|v + \varphi\|^2_{H^p(\Omega)}$ is a regularization term.

\begin{Proposition}
	For all  $\lambda > 1, \beta > 0$ and $\eta > 0$, for each $v \in H^p(\Omega),$ the functional 
	$J_v^{\lambda, \beta, \eta}: H_0 \to H_0$ has a unique minimizer.
	\label{prop_min}
\end{Proposition}
\begin{proof}
	It is obvious that $J_v^{\lambda, \beta, \eta}$ is coercive and weakly lower semicontinuous. Thus, it has a minimizer in $H_0$. The uniqueness of the minimizer can be deduced from the strict convexity of $J_v^{\lambda, \beta, \eta}$ in $H_0$.
\end{proof}

For each $n\geq 0$, thanks for Proposition \ref{prop_min}, we can minimize $J_{u_n}^{\lambda, \beta, \eta}(\varphi)$ on $H_0$. 
The unique minimizer is the desired function $h.$ 
We then set 
\begin{equation}
	u_{n + 1} = u_n + h.
\end{equation}
The construction of the sequence $\{u_n\}_{n \geq 0}$ above is summarized in Algorithm \ref{alg 1}. 
We will prove that the sequence $\{u_n\}_{n \geq 0}$  converges to $u^*$ in Section \ref{sec conv} as $n \to \infty$ and $\eta \to 0$.  
The presence of the Carleman weight $e^{2\lambda \mu_\beta(\x)}$ is a key point for us to prove this convergence result. 
\begin{algorithm}[h!]
\caption{\label{alg 1}The procedure to compute the numerical solution to \eqref{main eqn}}
	\begin{algorithmic}[1]
	\STATE \label{s1} 
	Choose a threshold number $0< \kappa_0 \ll 1.$ Choose an arbitrary initial solution $u_0 \in W.$
	\STATE Set $n = 0$.
		\STATE \label{step update} Solve the linear equation \eqref{2.3} for a function $h \in H_0$ by minimizing $J_{u_n}^{\lambda, \beta, \eta}(\varphi)$ on $H_0$.
		 Set $u_{n+1} = u_{n} + h$.
	\IF {$\|u_{n + 1} - u_n\|_{L^2(\Omega)} > \kappa_0$}
		\STATE Replace $n$ by $n + 1.$
		\STATE Go back to Step \ref{step update}.
	\ELSE	
		\STATE Set the computed solution $u_{\rm comp} = u_{n + 1}.$
	\ENDIF
	
\end{algorithmic}
\end{algorithm}


\section{The convergence analysis} \label{sec conv}

In this section, we prove that the sequence $\{u_n\}_{n \geq 0}$ generated by Algorithm \ref{alg 1} converges to the solution $u^*$ to \eqref{main eqn}.
The following result is the main theorem in this paper.
\begin{Theorem} \label{thm:main}
	Assume that $\|F\|_{C^2(\Omega, \R, \R^d)}$ is a finite number.
	Assume further that \eqref{main eqn} has a unique solution $u^*\in W$.
	Let $\{u_n\}_{n \geq 0}$ be the sequence generated by Algorithm \ref{alg 1}, where $u_0 \in W$ is chosen arbitrarily. 
	Then, there exist $\lambda_0 > 0$ and $\theta \in (0,1)$ such that, for all $\lambda > \lambda_0$,
	\begin{equation}
		\|u_{n + 1} - u^*\|_{\lambda, \mu, \beta}^2 \leq \theta^{n+1} \|u_0 - u^*\|_{\lambda, \mu, \beta}^2 + \eta \theta \frac{1 - \theta^{n+1}}{1-\theta} \|u^*\|_{H^p(\Omega)}^2.
		\label{3.3}
	\end{equation}
Here,
	\[
		\|v\|_{\lambda, \mu, \beta} = \Big[\int_{\Omega} e^{2\lambda \mu_\beta(\x)}\big( |v|^2 + |\nabla v|^2\big) \,d\x\Big]^{\frac{1}{2}} \quad \mbox{for all } v \in H^1(\Omega).
	\]
	\label{thm 3.1}
\end{Theorem}
\begin{proof}
	Fix $n \geq 0$.
	Let $h = u_{n+1} - u_n$.
	Due to Step \ref{step update} in Algorithm \ref{alg 1}, $h$ is the minimizer of $J_{u_n}^{\lambda, \beta, \eta}$. 
	By the variational principle, we have
	\begin{multline}
		\int_{\Omega} e^{2\lambda \mu_\beta(\x)} \big(-\Div(A \nabla h)- \Div(A \nabla  u_{n}) + F(\x,  u_n, \nabla  u_n) + D\mathcal F(u_n) h\big)
		\\
		\big(-\Div(A \nabla  \varphi) + D\mathcal F(u_n)\varphi\big)\,d\x 
		+ \eta\langle u_n + h, \varphi\rangle_{H^p(\Omega)} = 0
		\label{3.4}
	\end{multline}
for all $\varphi \in H_0.$
Since $u_{n+1} = u_n + h$, we can rewrite \eqref{3.4} as
\begin{multline}
		\int_{\Omega} e^{2\lambda \mu_\beta(\x)} \big(-\Div(A \nabla u_{n+1}) + F(\x,  u_n, \nabla  u_n) + D\mathcal F(u_n) (u_{n + 1} - u_n)\big)
		\\
		\big(-\Div(A \nabla  \varphi) + D\mathcal F(u_n)\varphi\big)\,d\x 
		+ \eta\langle u_{n+1}, \varphi\rangle_{H^p(\Omega)} = 0
		\label{3.5}
	\end{multline}
	for all $\varphi \in H_0.$
	As $u^*$ is the solution to \eqref{main eqn}, we have
	\begin{multline}
		\int_\Omega e^{2\lambda \mu_\beta(\x)} \big(-\Div(A \nabla  u^*) + F(\x, u^*, \nabla u^*)\big)
		\\
		\big(-\Div(A \nabla  \varphi) + D\mathcal F(u_n)\varphi\big)\, d\x 
	= 0
	\label{3.6}
	\end{multline}
for all $\varphi \in H_0.$
It follows from \eqref{3.5} and \eqref{3.6} that for all $\varphi \in H_0$
\begin{multline}
	\int_\Omega e^{2\lambda \mu_\beta(\x)}
	\big( - \Div(A \nabla (u_{n + 1} - u^*) )
		+ F(\x, u_n, \nabla u_n) - F(\x, u^*, \nabla u^*)
		\\
		+ D\mathcal F(u_n)(u_{n+1}  - u_n)
	\big)
	\big(-\Div(A \nabla  \varphi) + D\mathcal F(u_n)\varphi\big)\,d\x
	\\
	+ \eta\langle u_{n+1}, \varphi\rangle_{H^p(\Omega)} = 0.
	\label{3.7}
\end{multline}
Take $\varphi = u_{n+ 1} - u^* \in H_0$.
It follows from \eqref{3.7} that
\begin{multline}
	\int_{\Omega} e^{2\lambda \mu_\beta(\x)} 
	\Big[|\Div(A \nabla  \varphi)|^2
	-D\mathcal F(u_n) \varphi \,\Div(A \nabla  \varphi)
		-\big(F(\x, u_n, \nabla u_n) 
		\\
		- F(\x, u^*, \nabla u^*)\big)\Div(A \nabla \varphi) 		
		+ \big(F(\x, u_n, \nabla u_n) - F(\x, u^*, \nabla u^*)\big)D\mathcal F(u_n)\varphi
		\\
		-\Div(A \nabla  \varphi) D\mathcal F(u_n)(u_{n+1}  - u_n)
		+ D\mathcal F(u_n)(u_{n+1}  - u_n)D\mathcal F(u_n)\varphi
	\Big]
	\,d\x 
	\\
	= - \eta\langle u_{n+1}, \varphi\rangle_{H^p(\Omega)}.
\end{multline}
As $\|F\|_{C^2(\Omega, \R, \R^d)}=C<+\infty$, we can estimate
\begin{align*}
	&|D\mathcal F(u_n)\varphi| \leq C(|\varphi| + |\nabla \varphi|),\\
	&|F(\x, u_n, \nabla u_n) - F(\x, u^*, \nabla u^*)| \leq C(|u_n - u^*| + |\nabla(u_n - u^*)|),
	\\
	&|D\mathcal F(u_n)(u_{n+1}  - u_n)| = |D\mathcal F(u_n)(u_{n+1}  - u^* + u^* - u_n)|
	\\& \hspace{1cm} \leq C(|u_{n+1} - u^*| + |\nabla (u_{n+1} - u^*)|) + C(|u_{n} - u^*| + |\nabla (u_{n} - u^*)|),
\\&\mbox{and}
\\
	& - \eta\langle u_{n+1}, \varphi\rangle_{H^p(\Omega)} = - \eta\langle \varphi + u^*, \varphi\rangle_{H^p(\Omega)}
	\\
	&\hspace{2.9cm}= -\eta\|\varphi\|_{H^p(\Omega)}^2  -\eta \langle u^*, \varphi\rangle_{H^p(\Omega)} \leq \frac{\eta}{2}\|u^*\|_{H^p(\Omega)}^2.
\end{align*}
These estimates, together with the inequality $|ab| \leq 4a^2 + b^2/8$, imply
\begin{multline}
	\int_{\Omega} e^{2\lambda \mu_\beta(\x)} |\Div(A \nabla  \varphi)|^2\,d\x 
	\leq C
	\Big[
		\int_{\Omega}  e^{2\lambda \mu_\beta(\x)} (|\varphi|^2 + |\nabla \varphi|^2)\, d\x
		\\
		+ \int_{\Omega}  e^{2\lambda \mu_\beta(\x)} (|u_n - u^*|^2 + |\nabla (u_n - u^*)|^2) \,d\x
	\Big]
	+ \frac{\eta}{2}\|u^*\|_{H^p(\Omega)}^2.
	\label{3.9}
\end{multline}
Applying the Carleman estimate \eqref{Car est} for the function $\varphi$, we have
\begin{multline}
		\int_{\Omega} e^{2\lambda \mu_\beta(\x)}|\Div(A \nabla  \varphi)|^2\,d\x	
		\\
		\geq	
		C\lambda  \int_{\Omega}  e^{2\lambda \mu_\beta(\x)}|\nabla \varphi(\x)|^2\,d\x
		+ C\lambda^3  \int_{\Omega}   e^{2\lambda \mu_\beta(\x)}|\varphi(\x)|^2\,d\x.		
		\label{3.10}
	\end{multline}
Combining \eqref{3.9} and \eqref{3.10}, we have
\begin{multline}
	\lambda  \int_{\Omega}  e^{2\lambda \mu_\beta(\x)}|\nabla \varphi(\x)|^2\,d\x
		+ \lambda^3  \int_{\Omega}   e^{2\lambda \mu_\beta(\x)}|\varphi(\x)|^2\,d\x
		\\
		\leq
	C
	\Big[
		\int_{\Omega}  e^{2\lambda \mu_\beta(\x)} (|\varphi|^2 + |\nabla \varphi|^2) \,d\x
		+ \int_{\Omega}  e^{2\lambda \mu_\beta(\x)} (|u_n - u^*|^2 + |\nabla (u_n - u^*)|^2) \,d\x
	\Big]
	\\
	+ \frac{\eta}{2}\|u^*\|_{H^p(\Omega)}^2.
	\label{3.11}
\end{multline}
Letting $\lambda$ be sufficiently large,  we can simplify \eqref{3.11} as
\begin{multline}
	\lambda \int_{\Omega}  e^{2\lambda \mu_\beta(\x)} (|\varphi|^2 + |\nabla \varphi|^2) \,d\x
	\leq C \int_{\Omega}  e^{2\lambda \mu_\beta(\x)} (|u_n - u^*|^2 + |\nabla (u_n - u^*)|^2) \,d\x
	\\
	+ \frac{\eta}{2}\|u^*\|_{H^p(\Omega)}^2.
	\label{3.12}
\end{multline}
Recall that $\varphi = u_{n + 1} - u^*$. 
We get from \eqref{3.12} that
\begin{equation}
	\|u_{n + 1} - u^*\|_{\lambda, \mu, \beta}^2 \leq \frac{C}{\lambda} \|u_{n} - u^*\|_{\lambda, \mu, \beta}^2 + \frac{\eta}{2 \lambda}\|u^*\|_{H^p(\Omega)}^2.
	\label{3.13}
\end{equation}
Applying \eqref{3.13} for $n - 1$ and denoting $\theta = C/\lambda \in (0, 1)$,
we have
\begin{align*}
	\|u_{n + 1} - u^*\|_{\lambda, \mu, \beta}^2 
	&\leq \theta \Big[\theta \|u_{n-1} - u^*\|_{\lambda, \mu, \beta}^2 + \eta \theta\|u^*\|_{H^p(\Omega)}^2\Big] + \eta \theta\|u^*\|_{H^p(\Omega)}^2
	\\
	&= \theta^2   \|u_{n-1} - u^*\|_{\lambda, \mu, \beta}^2 + \eta \theta(1 + \theta)\|u^*\|_{H^p(\Omega)}^2.
\end{align*}
By induction, we have
\begin{equation}
	\|u_{n + 1} - u^*\|_{\lambda, \mu, \beta}^2 \leq \theta^{n+1} \|u_0 - u^*\|_{\lambda, \mu, \beta}^2 +\eta \theta \sum_{i = 0}^n \theta^n \|u^*\|_{H^p(\Omega)}^2,
\end{equation}
which implies \eqref{3.3}. 
The proof is complete.
\end{proof}

\begin{remark}[Removing the boundedness condition of $F$ in $C^2$ in Theorem \ref{thm 3.1}]
In the case when $\|F\|_{C^2(\Omega \times \R \times \R^{d + 1})} = \infty$, 
we need to assume that we know in advance that the true solution $u^*$ to \eqref{main eqn} belongs to the ball $B_M$ of $C^1(\Omega)$ for some $M>0$.
This assumption does not weaken the result since $M$ can be arbitrarily large.
Define the cut-off function $\chi \in C^\infty(\Omega \times \R \times \R^{d})$ as 
\[
	\chi(\x, s, \p) = \left\{
		\begin{array}{ll}
			1 & |s| + |\p| < M,\\
			0 & |s| + |\p| > 2 M
		\end{array}
	\right.
\]
and set $\widetilde F =  \chi F.$
Since $|u^*| + |\nabla u^*| < M$, it is obvious that $u^*$ solves the problem
\begin{equation}
\left\{
	\begin{array}{ll}
		 -\Div(A(\x) \nabla  u(\x))  +\widetilde F(\x, u(\x), \nabla u(\x)) = 0 &\x \in \Omega,\\
		u(\x) = f(\x) &\x \in \partial\Omega,\\
		\partial_\nu u(\x) = g(\x) &\x \in \partial \Omega.
	\end{array}
\right.
\label{aux eqn}
\end{equation}
Now, we can apply Algorithm \ref{alg 1}  for \eqref{aux eqn} to compute $u^*$. 
 \end{remark}
 
 \begin{remark}
 	Theorem \ref{thm 3.1} and estimate \eqref{3.3} rigorously guarantee that each sequence generated by Algorithm \ref{alg 1} converges to $u^*$ at the exponential rate. This fact is numerically confirmed by our numerical results in Section \ref{sec num}.
\label{rm 3.2}
 \end{remark}
 
 \begin{remark}
	As seen in  the proof of Theorem \ref{thm 3.1}, the efficiency of Algorithm \ref{alg 1} is guaranteed by Carleman estimate \eqref{Car est}.
Therefore, we call the proposed method described in Algorithm \ref{alg 1} the second generation of Carleman-based numerical methods.
This second generation includes the method in \cite{LeNguyen:jiip2022, NguyenKlibanov:ip2022} in which the iteration scheme is developed based on the contraction principle and Carleman estimates.
The first generation of 		
		 Carleman-based numerical method was developed in \cite{KlibanovIoussoupova:SMA1995}, which is called the convexification.
	See \cite{KlibanovNik:ra2017, KlibanovNguyenTran:JCP2022, LeNguyen:jiip2022, LeNguyen:JSC2022} for some following up results.
	Like the convexification method, Algorithm \ref{alg 1} can be used to compute solutions to nonlinear PDEs without requesting an initial good guess.
	The advantage of our new method is the fast convergence rate, see Remark \ref{rm 3.2}.
\end{remark}

\section{Numerical study} \label{sec num}

In this section, we present some numerical results obtained by Algorithm \ref{alg 1}.
For simplicity, we set $d = 2$ and $\Omega = (-1, 1)^2$. 
On $\Omega,$ we arrange an $N \times N$ grid points
\[
	\mathcal G = \{(x_i = -1 + (i - 1)\delta, y_j = -1 + (j - 1)\delta): 1 \leq i, j \leq N\}
\] 
where $\delta  = 2/(N-1).$ In our numerical scripts, $N = 80$.

In Step \ref{s1} of Algorithm \ref{alg 1}, we choose a function $u_0 \in W$. 
It is natural to find a function $u_0$ satisfying the equation obtained by removing from \eqref{main eqn} the nonlinearity $F(\x, u, \nabla u)$. 
We apply the Carleman-based quasi-reversibility method to do so.
That means, $u_0$ is the minimizer of
\begin{equation}
	J_0^{\lambda, \beta, \eta}(\varphi) =  \int_{\Omega} e^{\lambda \mu_\beta(\x)}|\Div(A \nabla  \varphi)|^2 d\x+ \eta\|\varphi\|^2_{H^2(\Omega)}
	\label{J0}
\end{equation}
subject to the boundary conditions in \eqref{bdry}.
To simplify the efforts in implementation, the norm in the regularization term is the $H^2(\Omega)$-norm rather than the $H^p(\Omega)$-norm. This change does not affect the performance of Algorithm \ref{alg 1}. Algorithm \ref{alg 1} still  provides satisfactory solutions to \eqref{main eqn}.
\begin{remark}
	We employ the  Carleman-based quasi-reversibility method to find $u_0$  for the consistency to Step \ref{step update} of Algorithm 1. 
	Since $u_0 \in W$ can be chosen arbitrarily, one can use the quasi-reversibility method without the presence of the Carleman weight function $e^{\lambda \mu_\beta(\x)}$. 
	In our computation for all numerical examples below,   $\lambda = 4$, $\beta = 10$ and $\x_0 = (-4, 0).$
	The regularized parameter is $\eta = 10^{-4}$. The threshold number $\kappa_0 = 10^{-6}.$
\end{remark}

We minimize $J_0^{\lambda, \beta, \eta}$ on $H_0$ by the least square MATLAB command ``lsqlin". 
The implementation for the quasi-reversibility method to minimize more general functionals than $J_0^{\lambda, \beta, \eta}$ was described in \cite[\S 5.3]{LeNguyen:jiip2022} and in \cite[\S 5]{Nguyen:CAMWA2020}. We do not repeat this process here. 
In Step \ref{step update} of Algorithm \ref{alg 1}, given $u_n \in W,$ we minimize the functional $J^{\lambda, \beta, \eta}_{u_n}$ on $H_0$. 
Again, we refer the reader to \cite[\S 5.3]{LeNguyen:jiip2022} and in \cite[\S 5]{Nguyen:CAMWA2020} for details in implementation.
The scripts for other steps of Algorithm \ref{alg 1} can be written easily.

\subsection{Quasilinear elliptic equations}

The convexification method, first introduced in \cite{KlibanovIoussoupova:SMA1995}, was used to numerically  solve quasilinear elliptic equations in \cite{KlibanovNik:ra2017, KlibanovNguyenTran:JCP2022, LeNguyen:JSC2022}.
Our approach here is of course different based on iteration and linearization.
In particular, Step \ref{step update} in Algorithm \ref{alg 1} is very efficient as we only need to solve the linear equation \eqref{2.3} as opposed to solving directly a nonlinear quasilinear elliptic equation.
In this subsection, we present two (2) numerical tests. In both tests, we choose the matrix $A$ to be \[A = \left[
	\begin{array}{ll}
		2&1\\1&2
	\end{array}
\right].\]
That means, $\Div(A \nabla u) = 2u_{xx} + 2u_{xy} + 2u_{yy}.$

\medskip

In test 1, we solve \eqref{main eqn} when 
\begin{equation}
	F(\x, s, \p) = s + |\p| - \big(- x^2 + 2y^2 + \sqrt{4x^2 + 16y^2} - 4\big)
	\label{F1noHJ}
\end{equation}
for all $\x = (x, y) \in \Omega$, $s \in \R$ and $\p \in \R^2$.
The boundary conditions are given by
\begin{equation}
	u(\x) = -x^2 + 2y^2, \quad \partial_\nu u(\x) = \langle -2x, 4y\rangle \cdot \nu
	\label{Cauchy1}
\end{equation}
for all $\x = (x, y) \in \partial \Omega.$
The true solution of \eqref{main eqn} is the function $u_{\rm true}(x, y) = -x^2 + 2y^2$. 

\begin{figure}[h!]
	\subfloat[The function $u^*$]{\includegraphics[width = .3\textwidth]{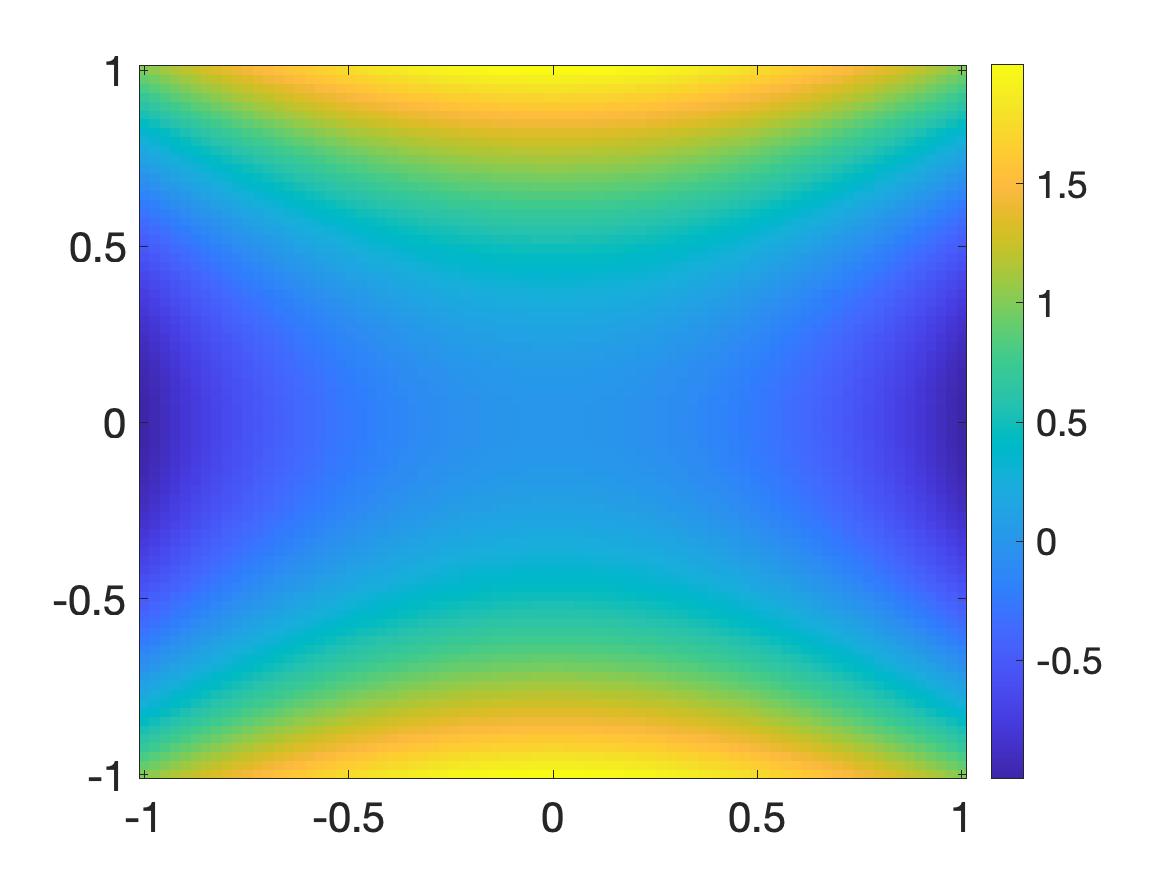}}
	\quad
	\subfloat[ The relative error $\frac{|u^*(\x) - u_{\rm comp}(\x)|}{\|u_{\rm true}\|_{L^{\infty}(\Omega)}}$]{\includegraphics[width = .3\textwidth]{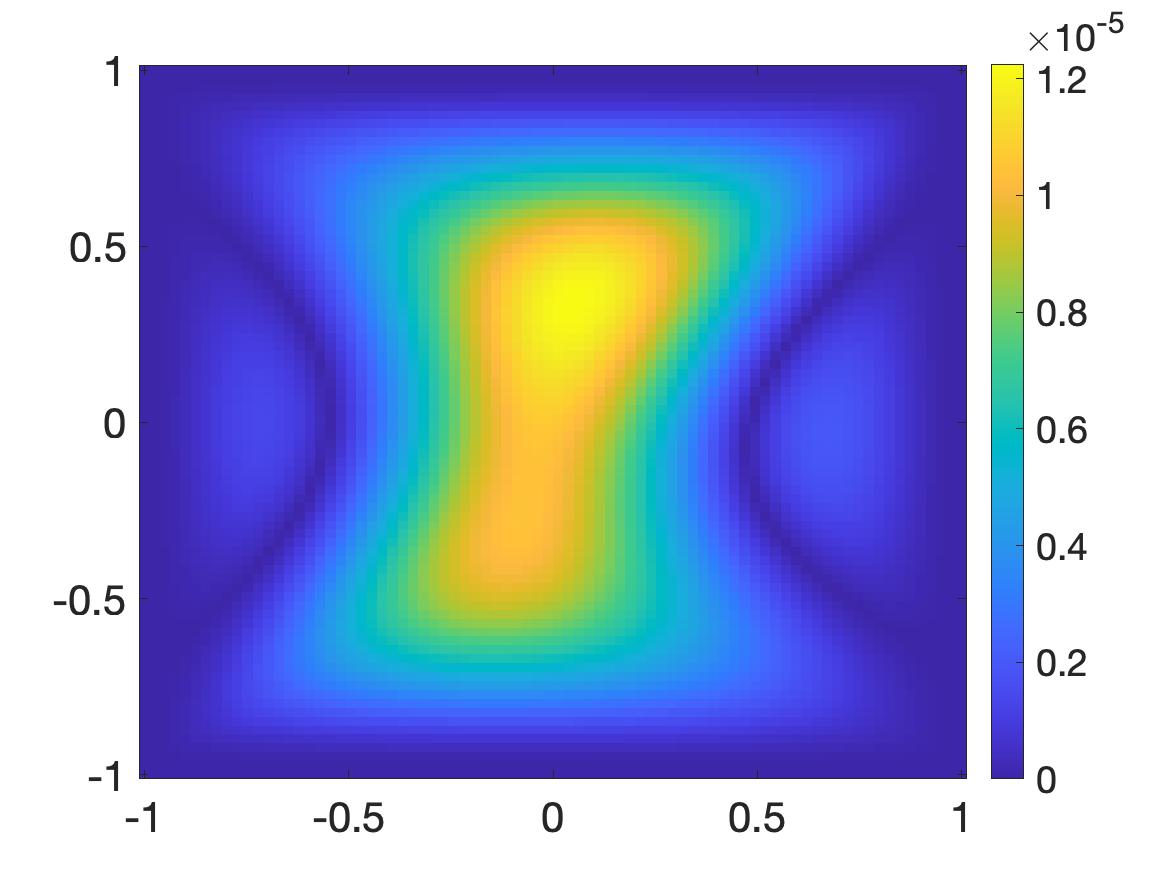}}
	\quad
	\subfloat[\label{1c} $\|u_n - u_{n - 1}\|_{L^{2}(\Omega)}.$ The horizontal axis is the number of iteration $n.$]{\includegraphics[width = .3\textwidth]{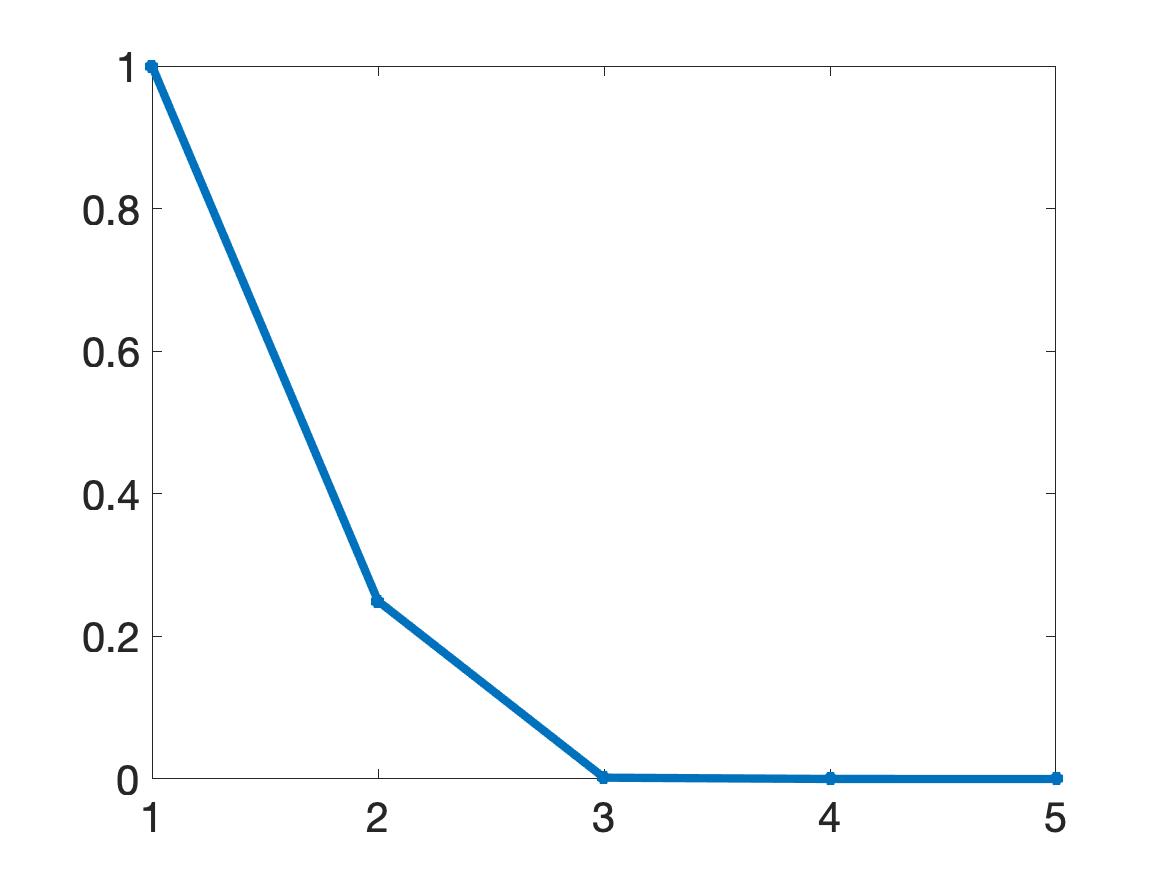}}
	\caption{\label{noHJ1} Test 1. The numerical solution to \eqref{main eqn} where $F$ is given in \eqref{F1noHJ} and the boundary data are given in \eqref{Cauchy1}.}
	\end{figure}
It is evident from Figure \ref{noHJ1} that Algorithm \ref{alg 1} provides out of expectation solution for test 1. The relative error 
$\|u_{\rm true} - u_{\rm comp}\|_{L^{\infty}(\Omega)}/\|u_{\rm true}\|_{L^{\infty}(\Omega)} = 1.23 \times 10^{-5}$.
One can see from Figure \ref{1c} that Algorithm \ref{alg 1} converges at the third iteration. 

\medskip

In test 2, we solve \eqref{main eqn} when 
\begin{multline}
	F(\x, s, \p) =  |\p| - \Big[\sqrt{\Big[\frac{\pi}{2} \cos\big(\frac{\pi}{2}(x + y)\big)  + e^x\Big]^2 + \frac{\pi^2}{4} \cos^2\big(\frac{\pi}{2}(x + y)\big)} 
	\\
	+ \frac{3\pi^2}{2} \sin\big(\frac{\pi}{2}(x + y)\big) - 2e^x 
	\Big]
	\label{F2noHJ}
\end{multline}
for all $\x = (x, y) \in \Omega$, $s \in \R$ and $\p \in \R^2$.
The boundary conditions are given by
\begin{equation}
	u(\x) = \sin\big(\frac{\pi}{2}(x + y)\big) + e^x, \quad 
	\partial_\nu u(\x) = \big\langle \frac{\pi}{2} \cos\big(\frac{\pi}{2}(x + y)\big) + e^x, \cos\big(\frac{\pi}{2}(x + y)\big) \big\rangle \cdot \nu
		\label{Cauchy2}
\end{equation}
for all $\x = (x, y) \in \partial \Omega.$
The true solution of \eqref{main eqn} is the function $u_{\rm true}(x, y) =\sin\big(\frac{\pi}{2}(x + y)\big) + e^x$. 

\begin{figure}[h!]
	\subfloat[The function $u^*$]{\includegraphics[width = .3\textwidth]{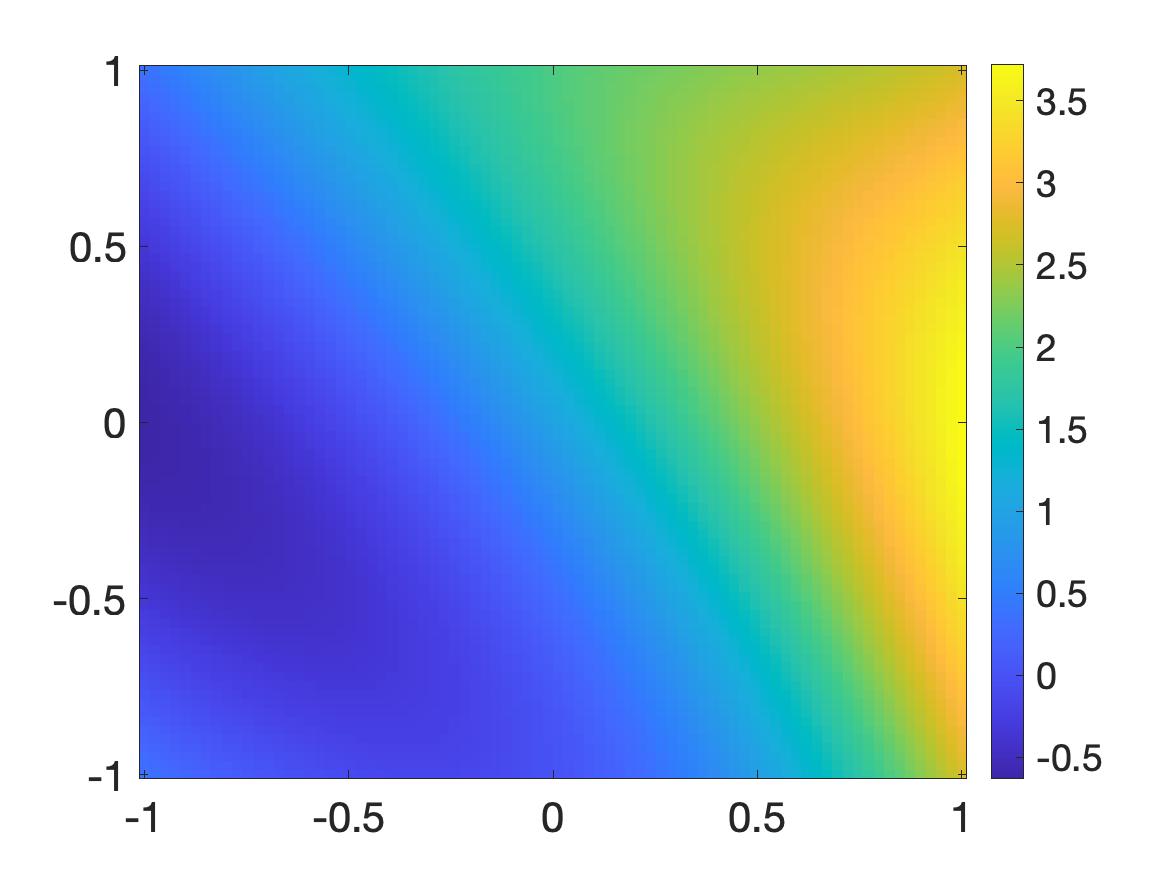}}
	\quad
	\subfloat[ The relative error $\frac{|u^*(\x) - u_{\rm comp}(\x)|}{\|u_{\rm true}\|_{L^{\infty}(\Omega)}}$]{\includegraphics[width = .3\textwidth]{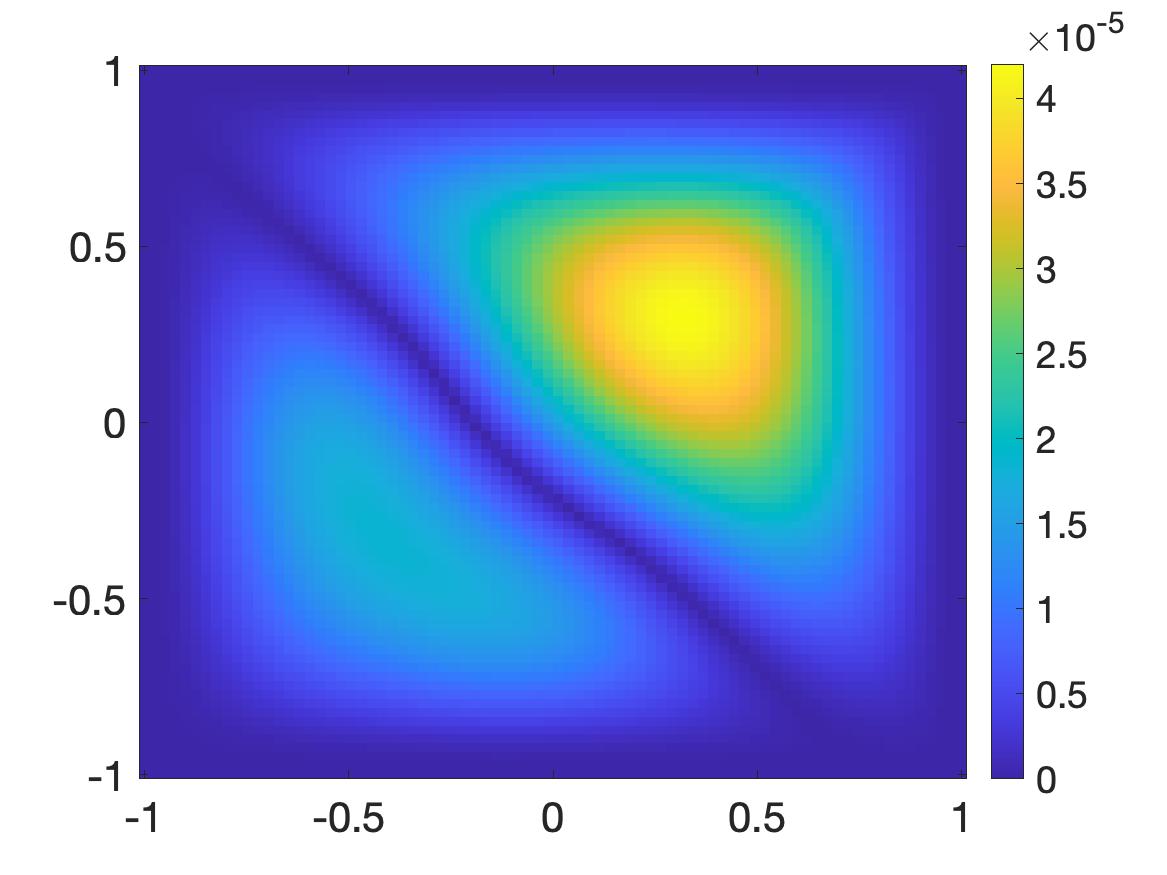}}
	\quad
	\subfloat[\label{2c} $\|u_n - u_{n - 1}\|_{L^{2}(\Omega)}$. The horizontal axis is the number of iteration $n$]{\includegraphics[width = .3\textwidth]{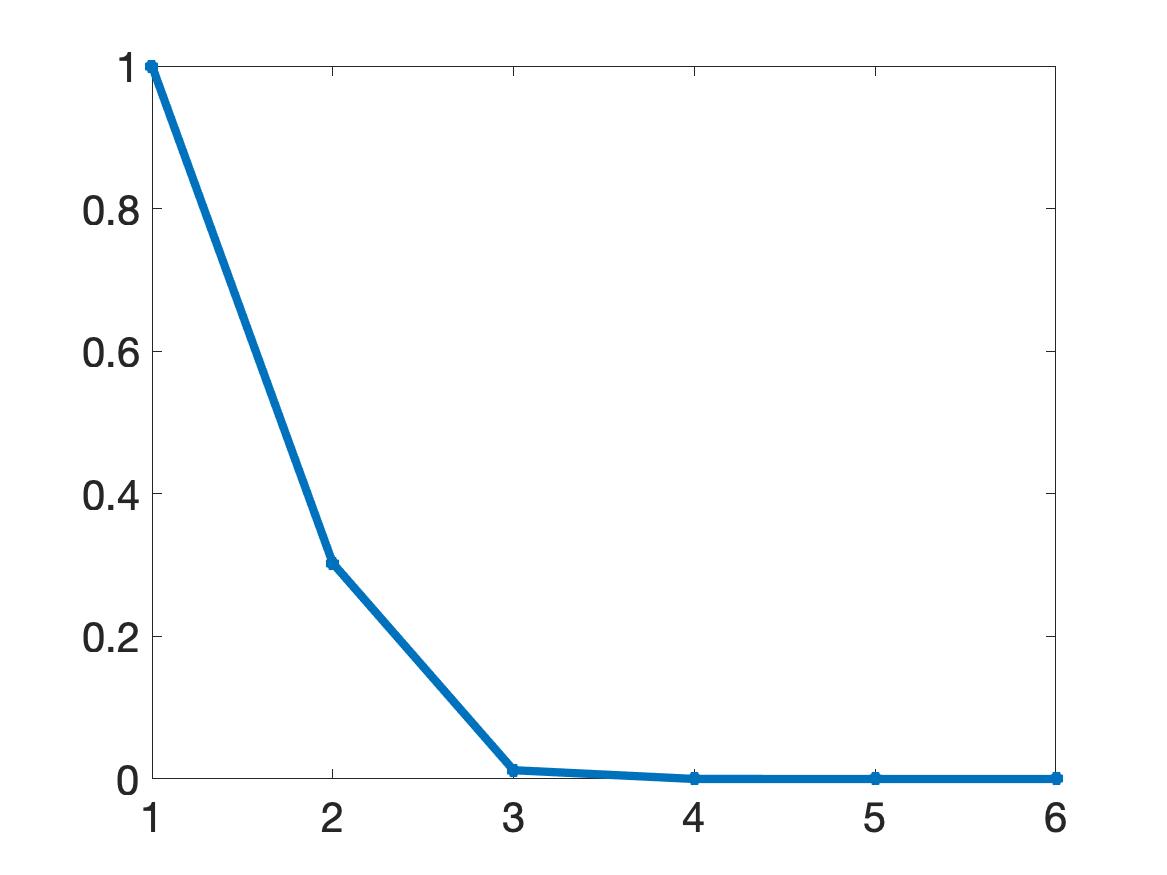}}
	\caption{\label{noHJ2} Test 2. The numerical solution to \eqref{main eqn} where $F$ is given in \eqref{F2noHJ} and the boundary data are given in \eqref{Cauchy2}.}
\end{figure}
It is evident from Figure \ref{noHJ2} that Algorithm \ref{alg 1} provides out of expectation solution for test 2. 
The relative error $\|u_{\rm true} - u_{\rm comp}\|_{L^{\infty}(\Omega)}/\|u_{\rm true}\|_{L^{\infty}(\Omega)} = 4.19 \times 10^{-5}$.
One can see from Figure \ref{2c} that Algorithm \ref{alg 1} converges at the fifth iteration.

\subsection{Application to first-order Hamilton-Jacobi equations}
Our aim in this subsection is to solve numerically 
\begin{equation}	 F(\x, u, \nabla u) = 0 \quad
	\mbox{for all } \x \in \Omega.
	\label{HJ without vis}
\end{equation}
with the boundary conditions 
\begin{equation}
	u|_{\partial \Omega} = f 
	\quad 
	\mbox{and}
	\quad \partial_{\nu} u|_{\partial \Omega} = g.
	\label{bdry}
\end{equation}
Basically, we use the vanishing viscosity process to approximate solutions to \eqref{HJ without vis}.
For $\epsilon>0$, consider
\begin{equation}
	-\epsilon \Delta u + F(\x, u, \nabla u) = 0 \quad
	\mbox{for all } \x \in \Omega
	\label{HJ}
\end{equation}
with boundary conditions \eqref{bdry}.
Again, \eqref{HJ}--\eqref{bdry} is an over-determined boundary value problem.
For $\epsilon>0$ sufficiently small, assume that \eqref{HJ}--\eqref{bdry} has a solution $u^\epsilon \in W$.
Then, $u^\epsilon$ approximates $u$, solution to \eqref{HJ without vis}--\eqref{bdry}, quite well under some appropriate assumptions on $F$.
In our numerical tests, we choose $\epsilon=\epsilon_0=10^{-3}$.

In this part, we provide six (6) numerical tests, in which we compute the viscosity solution to some Hamilton-Jacobi equations of the form \eqref{HJ without vis} given Cauchy boundary data. 
That means, by applying Algorithm \ref{alg 1}, we numerically find a function $u_{\rm comp}$ satisfying \eqref{HJ}-\eqref{bdry} when $F$, $f$ and $g$ are given.
The verification that $u^*$ is the correct viscosity solution can be done in a similar manner as that in \cite{Tran19, KlibanovNguyenTran:JCP2022}.

\medskip

{\it Test 1.}
In this test, we solve \eqref{HJ}-\eqref{bdry} when 
\begin{equation}
	F(\x, s, \p) = s + |\p| + |x| - 1. 
	\label{F1}
\end{equation}
for all $\x = (x, y) \in \Omega, s \in \R, \p \in \R^2$.
The boundary conditions are given by
\begin{equation}
	u(\x) = -|x|,
	\quad 
	\partial_\nu u(\x) =\langle -\mbox{sign}(x), 0\rangle \cdot \nu
\label{boundary1}
\end{equation} 
for all $\x = (x, y) \in \partial \Omega$.
In this case, the true solution is $u^* = -|x|.$ 
The numerical result is displayed in Figure \ref{fig 1}.

\begin{figure}[h!]
	\subfloat[The true solution $u^*$ to the Hamilton-Jacobi equation where the Hamiltonian is given in \eqref{F1}.]{\includegraphics[width=.3\textwidth]{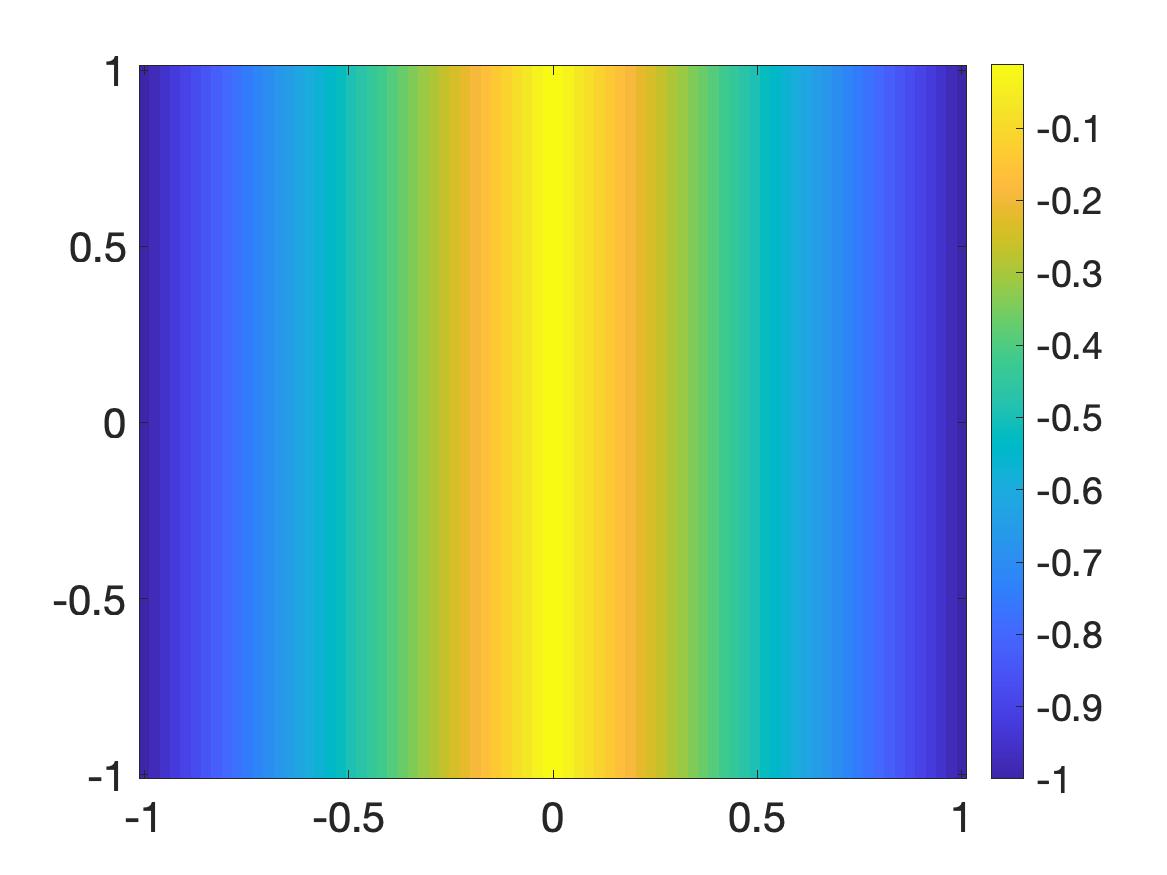}}
	\quad
	\subfloat[The initial solution $u_0$ computed by minimizing $J_0^{\lambda, \beta, \eta}$ defined in \eqref{J0}]{\includegraphics[width=.3\textwidth]{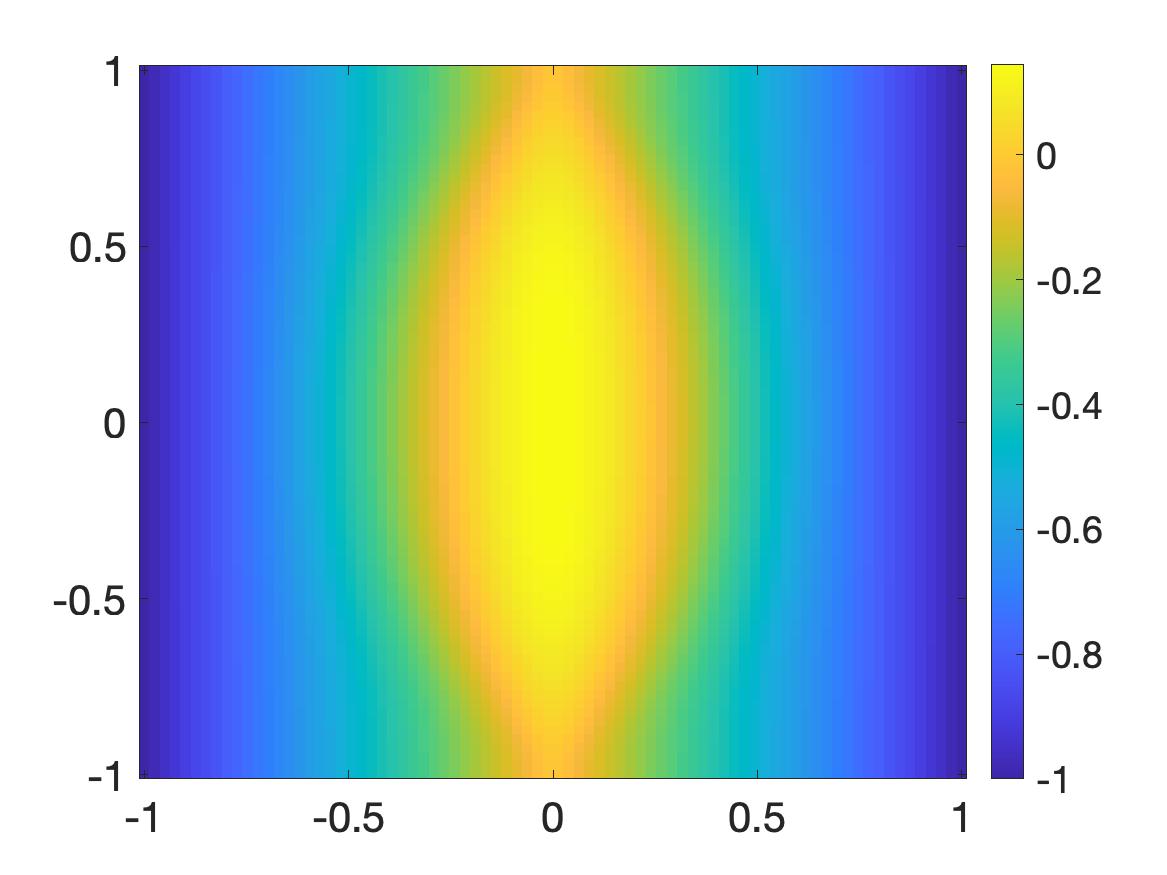}}
	\quad
	\subfloat[The computed solution $u_{\rm comp}$.]{\includegraphics[width=.3\textwidth]{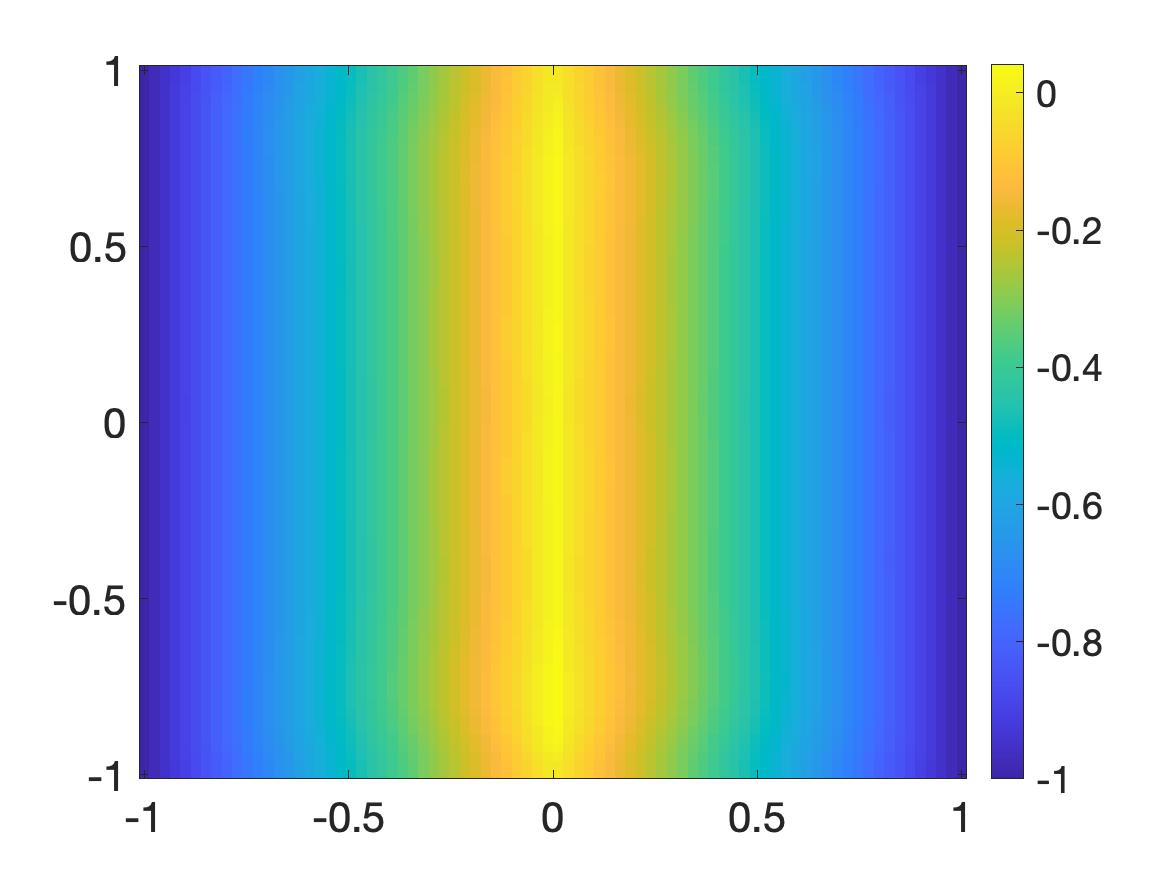}}
	
	\subfloat[$\|u_{n} - u_{n - 1}\|_{L^2(\Omega)}$. The horizontal axis is the number of iteration $n$.]{\includegraphics[width=.3\textwidth]{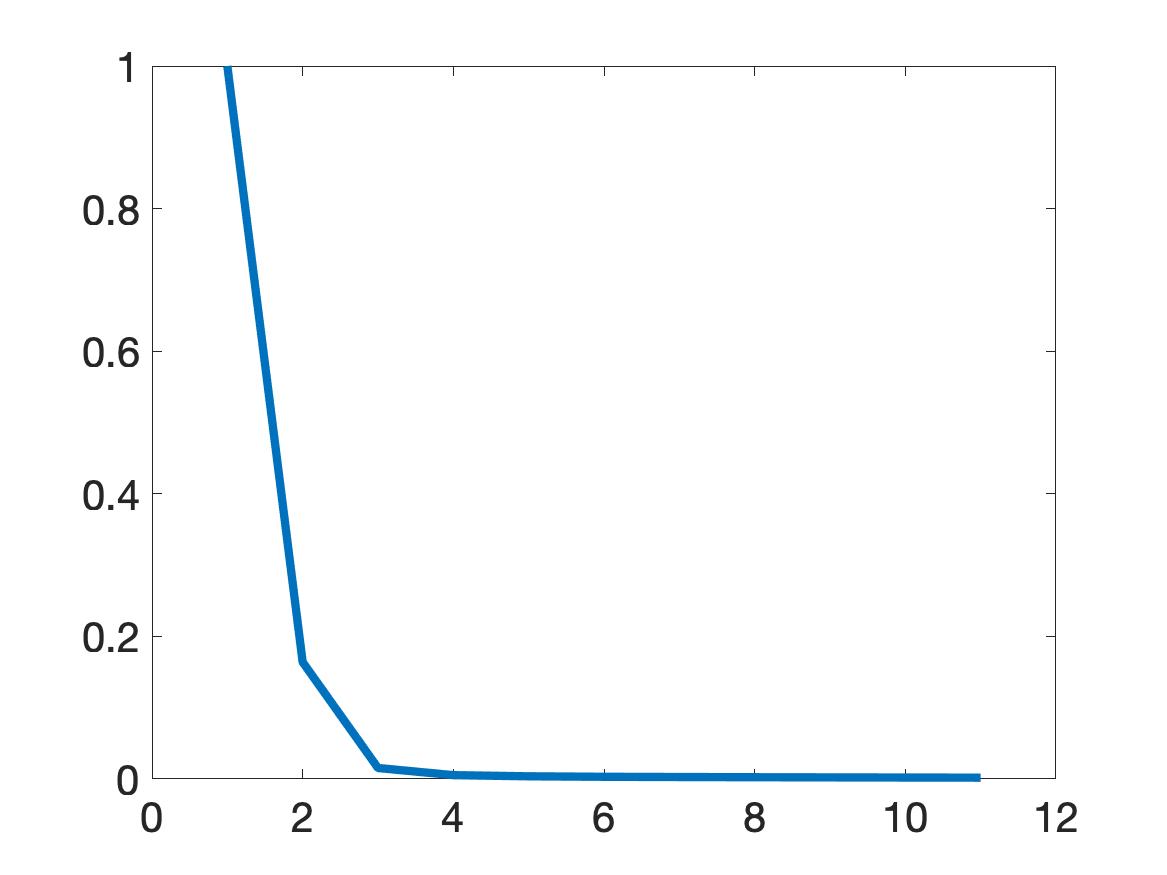}}
	\quad
	\subfloat[The true and computed solutions on the line $\{(x = 0 , y) \in \Omega\}$]{\includegraphics[width=.3\textwidth]{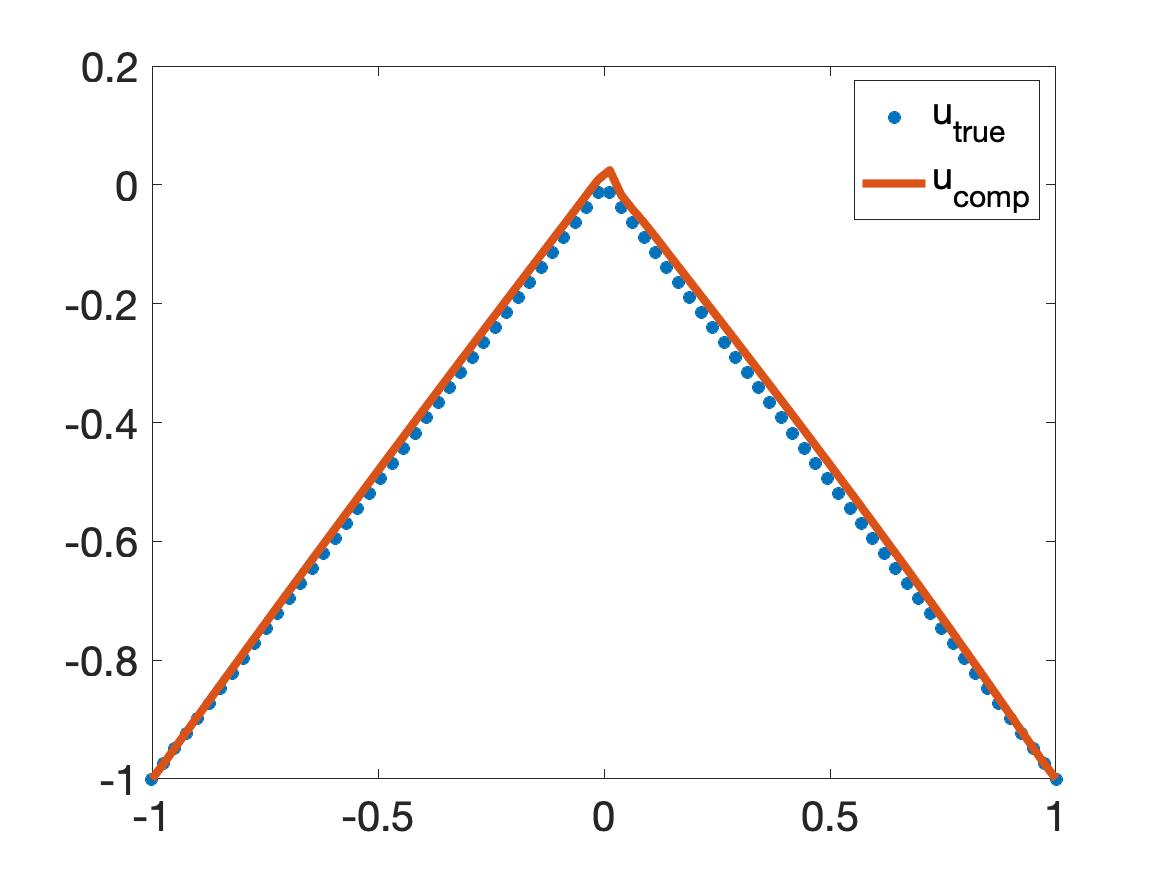}}
	\quad
	\subfloat[ The relative error $\frac{|u^*(\x) - u_{\rm comp}(\x)|}{\|u_{\rm true}\|_{L^{\infty}(\Omega)}}.$]{\includegraphics[width = .3\textwidth]{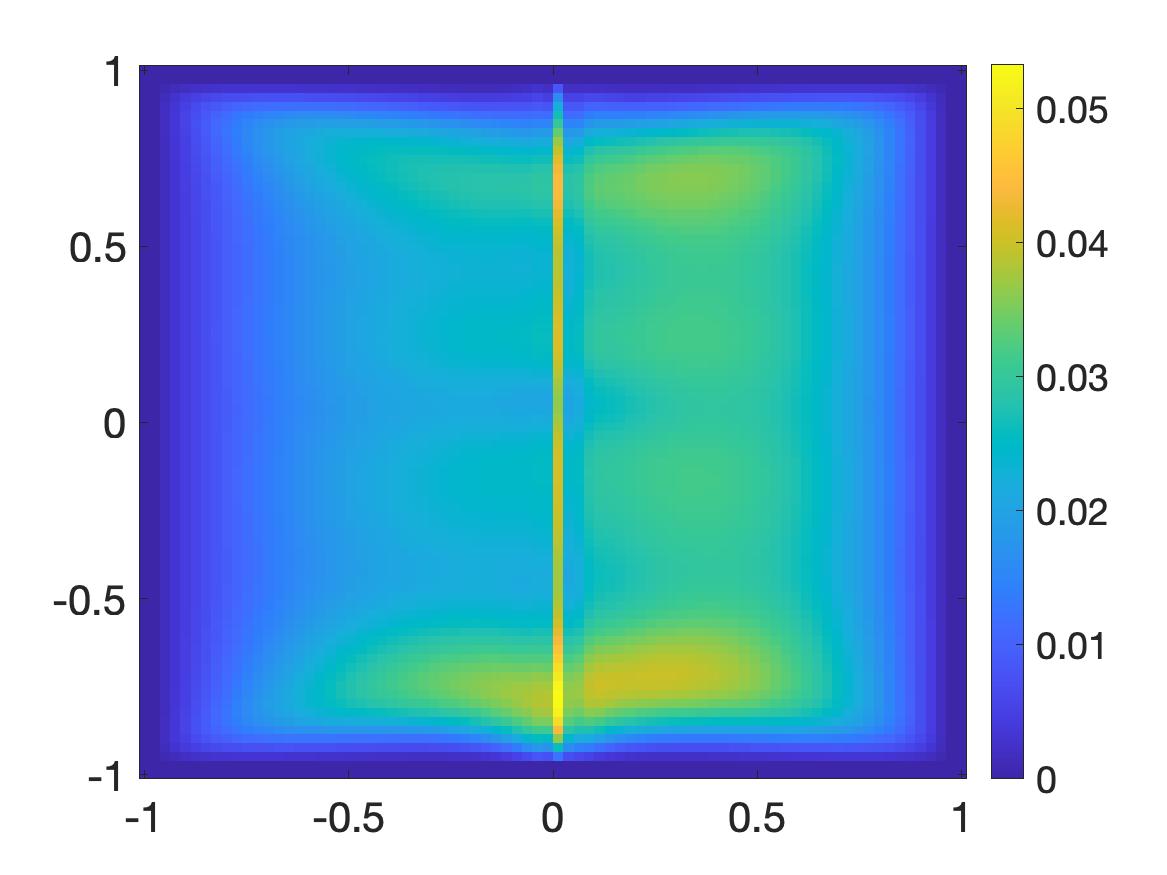}}
	\caption{\label{fig 1} Test 1. The true viscosity solution to \eqref{HJ}-\eqref{bdry} and the computed one. The Hamiltonian and the boundary conditions are given in \eqref{F1}-\eqref{boundary1}.}
\end{figure}

It is evident from Figure \ref{fig 1} that
we successfully compute the  solution $u_{\rm comp}$. 
The procedure described in Algorithm \ref{alg 1} converges after four iterations. 
The relative error $\frac{\|u^* - u_{\rm comp}\|_{L^\infty(\Omega)}}{\|u^*\|_{L^\infty(\Omega)}} = 5.33\%.$

\medskip

{\it Test 2.} We find the viscosity solution to the {\it eikonal} equation. 
In this test, we solve \eqref{HJ}-\eqref{bdry} when the Hamiltonian is 
\begin{equation}
	F(\x, s, \p) = |\p|^2 - (1 + (1 + \mbox{sign}(x + y))^2)
	\label{F2}
\end{equation}
for all $\x = (x, y) \in \Omega, s \in \R, \p \in \R^2$.
The boundary conditions are given by
\begin{equation}
	u(\x) = -|x + y| - y,
	 \quad  \partial_\nu u(\x) =( -\mbox{sign}(x+ y), -\mbox{sign}(x+ y) - 1 ) \cdot \nu
	 \label{boundary2}
\end{equation}
 for all $\x = (x, y) \in \partial \Omega$.
 The true solution is $u^*(\x) = -|x + y| - y$.
 The graphs of $u^*$ and $u_{\rm comp}$ are displayed in Figure \ref{fig 2}.
 \begin{figure}[h!]
	\subfloat[The true solution $u^*$ to the Hamilton-Jacobi equation where the Hamiltonian is given in \eqref{F2}.]{\includegraphics[width=.3\textwidth]{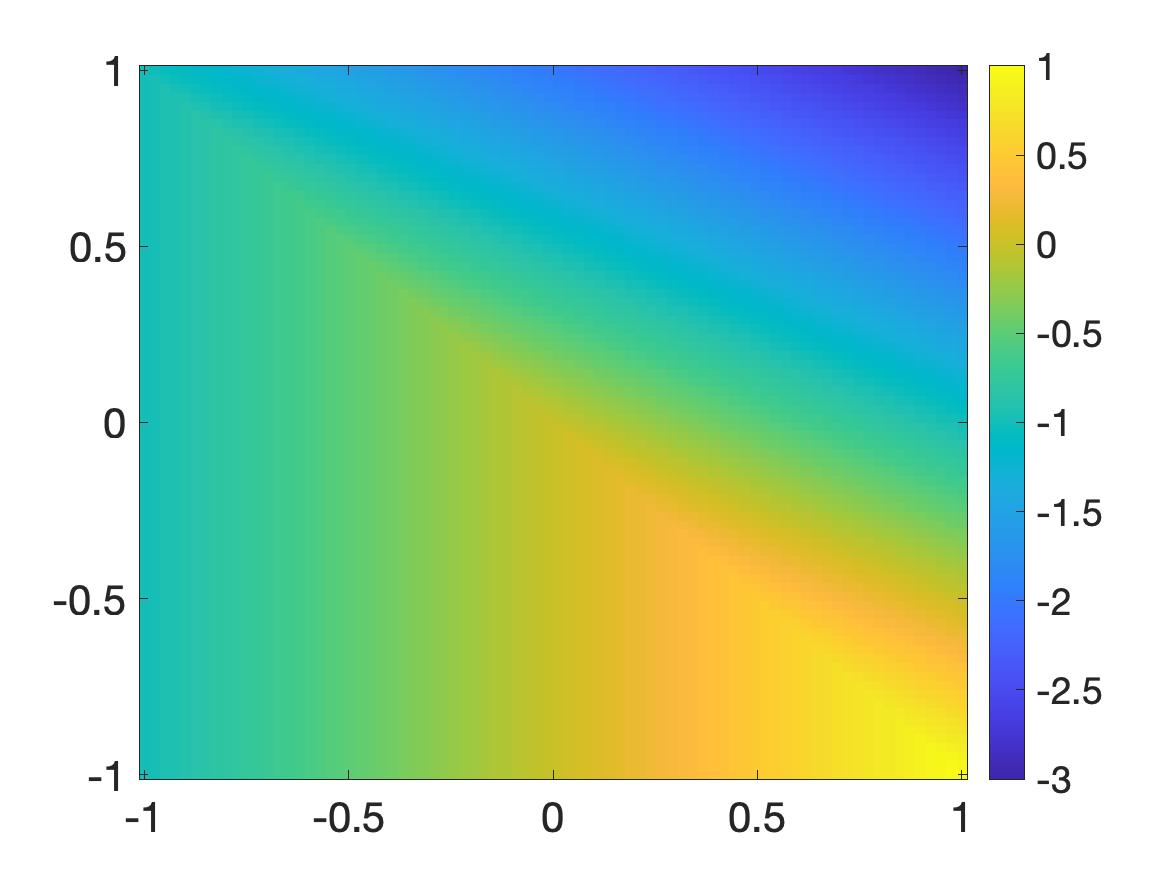}}
	\quad
	\subfloat[The initial solution $u_0$ computed by minimizing $J_0^{\lambda, \beta, \eta}$ defined in \eqref{J0}]{\includegraphics[width=.3\textwidth]{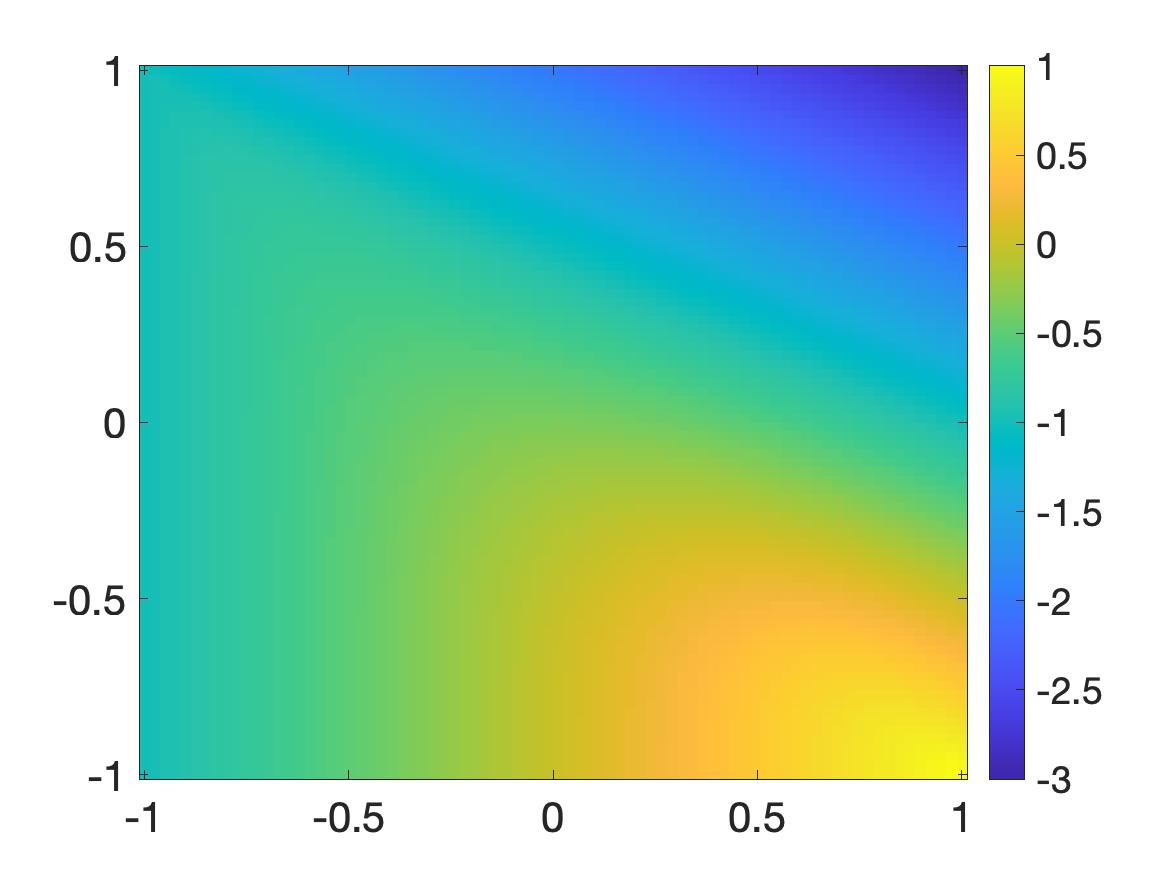}}
	\quad
	\subfloat[The computed solution $u_{\rm comp}$.]{\includegraphics[width=.3\textwidth]{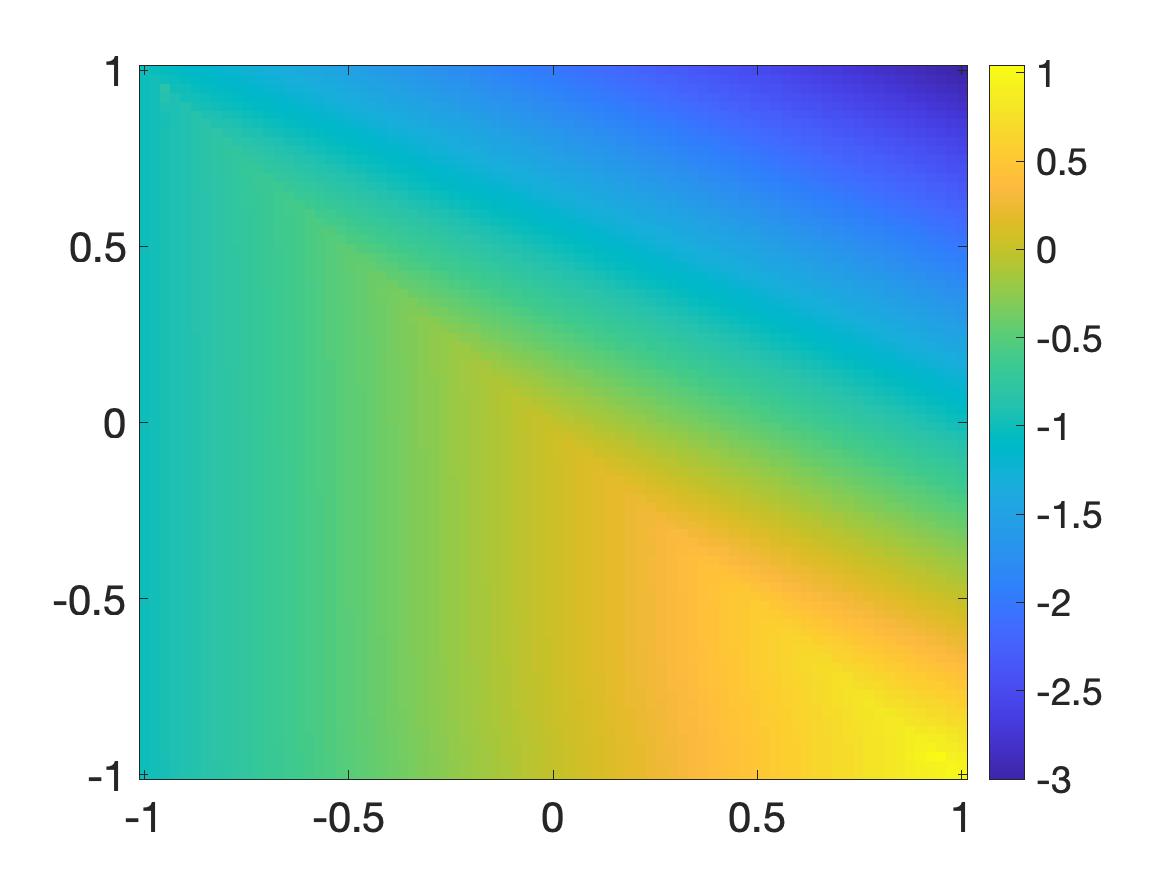}}

	\subfloat[$\|u_{n} - u_{n - 1}\|_{L^2(\Omega)}$. The horizontal axis is the number of iteration $n$.]{\includegraphics[width=.3\textwidth]{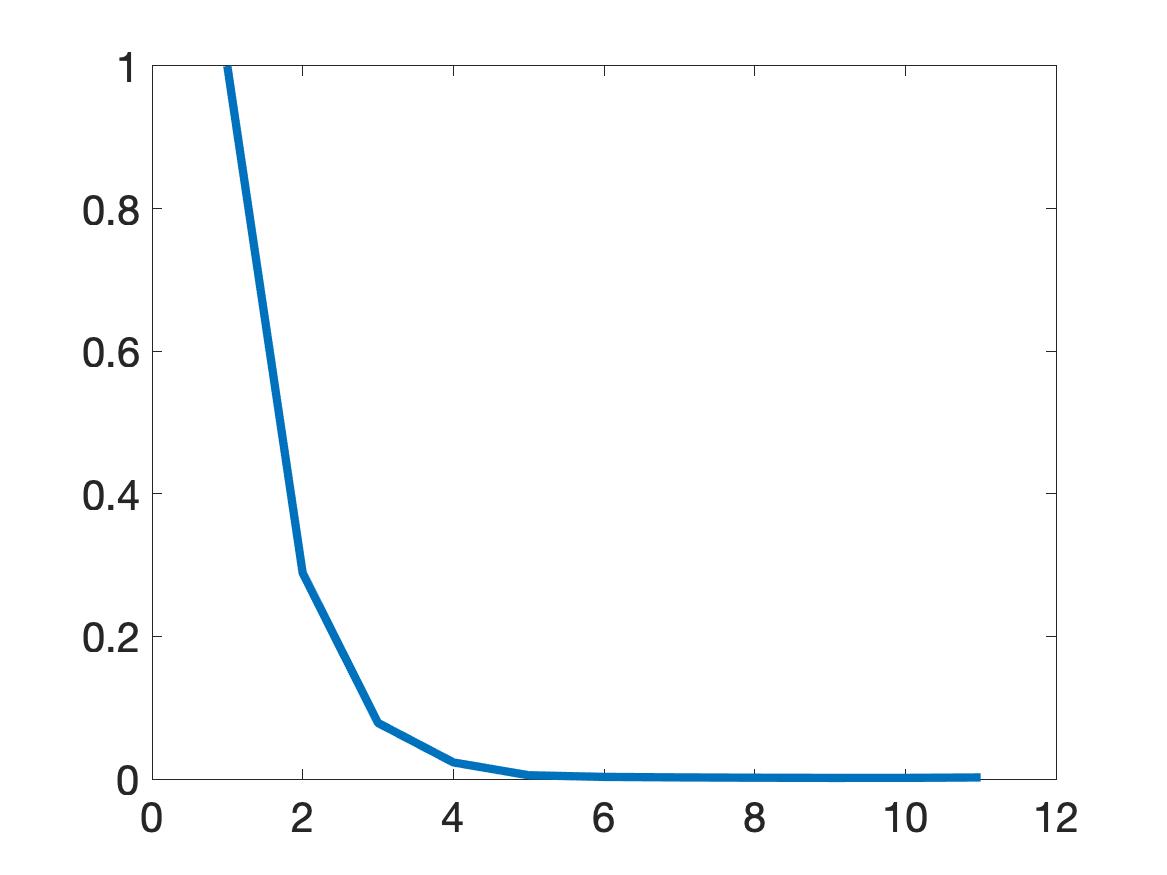}}
	\quad
	\subfloat[The true and computed solutions on the line $\{(x = 0.5 , y) \in \Omega\}$]{\includegraphics[width=.3\textwidth]{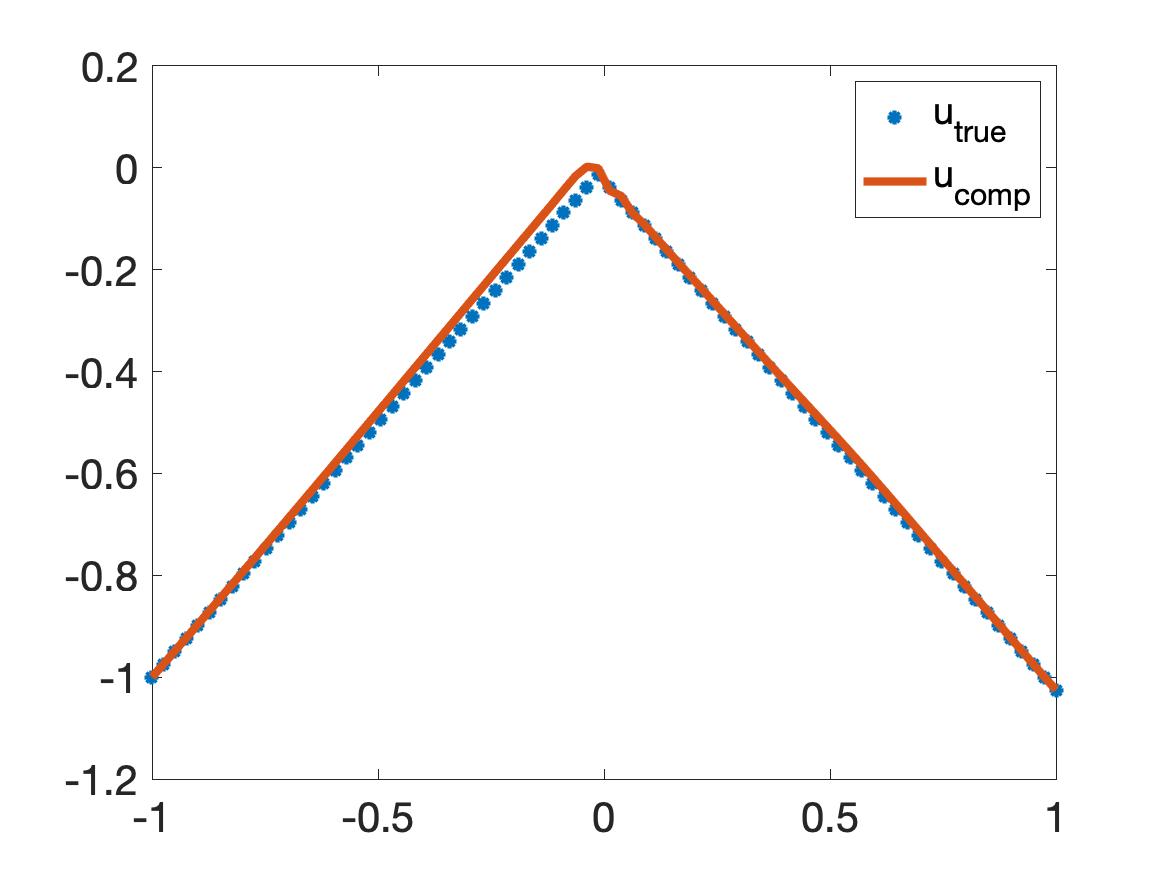}}
	\quad
	\subfloat[ The relative error $\frac{|u^*(\x) - u_{\rm comp}(\x)|}{\|u_{\rm true}\|_{L^{\infty}(\Omega)}}.$]{\includegraphics[width = .3\textwidth]{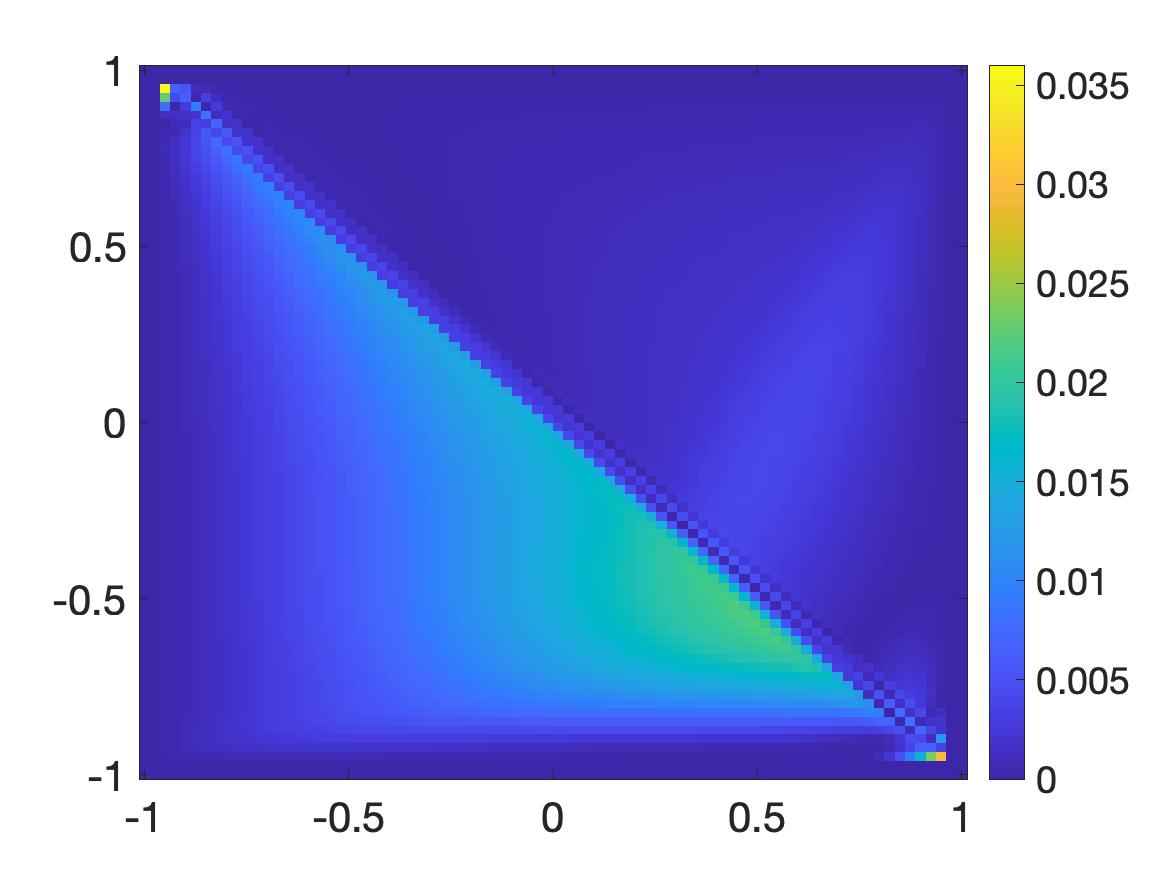}}
	\caption{\label{fig 2} Test 2. The true viscosity solution to \eqref{HJ}-\eqref{bdry} and the computed one. The Hamiltonian and the boundary conditions are given in \eqref{F2}-\eqref{boundary2}.}
\end{figure}
The relative error $\frac{\|u^* - u_{\rm comp}\|_{L^\infty(\Omega)}}{\|u^*\|_{L^\infty(\Omega)}} = 3.6\%.$
 
 \medskip
 
{\it Test 3.} We test the case when the Hamiltonian is not convex with respect  its third variable. 
The Hamiltonian in this test is given by
\begin{multline}
	F(\x, s, \p) = 20 s + |p_1| - |p_2|
	- \Big[
		20(-|x + 0.5| + e^{\cos(2\pi(x^2 + y^2))})
		\\ 
		+ |\mbox{sign}(x + 0.5) + 4\pi x\sin(2\pi(x^2 + y^2))e^{\cos(2\pi(x^2 + y^2))}|
		\\
		- |4\pi y\sin(2\pi(x^2 + y^2))e^{\cos(2\pi(x^2 + y^2))}|
	\Big]
	\label{F3}
\end{multline}
for all $\x = (x, y) \in \Omega, s \in \R, \p = (p_1, p_2) \in \R^2$.
The boundary conditions are given by
\begin{equation}
	u(\x) = -|x + 0.5| + e^{\cos(2\pi(x^2 + y^2))}
	\quad \mbox{for all } \x = (x, y) \in \partial \Omega
	\label{boundary31}
\end{equation}
and
\begin{multline}
	 \partial_\nu u(\x) = \big(-\mbox{sign}(x + 0.5) -4\pi x\sin(2\pi(x^2 + y^2))e^{\cos(2\pi(x^2 + y^2))},
	 \\ -4\pi y\sin(2\pi(x^2 + y^2))e^{\cos(2\pi(x^2 + y^2))}\big) \cdot \nu
	 \label{boundary32}
\end{multline}
 for all $\x = (x, y) \in \partial \Omega$.
 The true solution is $u^*(\x) = -|x + 0.5| + e^{\cos(2\pi(x^2 + y^2))}$.
The graphs of $u^*$ and $u_{\rm comp}$ are displayed in Figure \ref{fig 3}.
 \begin{figure}[h!]
	\subfloat[The true solution $u^*$ to the Hamilton-Jacobi equation where the Hamiltonian is given in \eqref{F3}.]{\includegraphics[width=.3\textwidth]{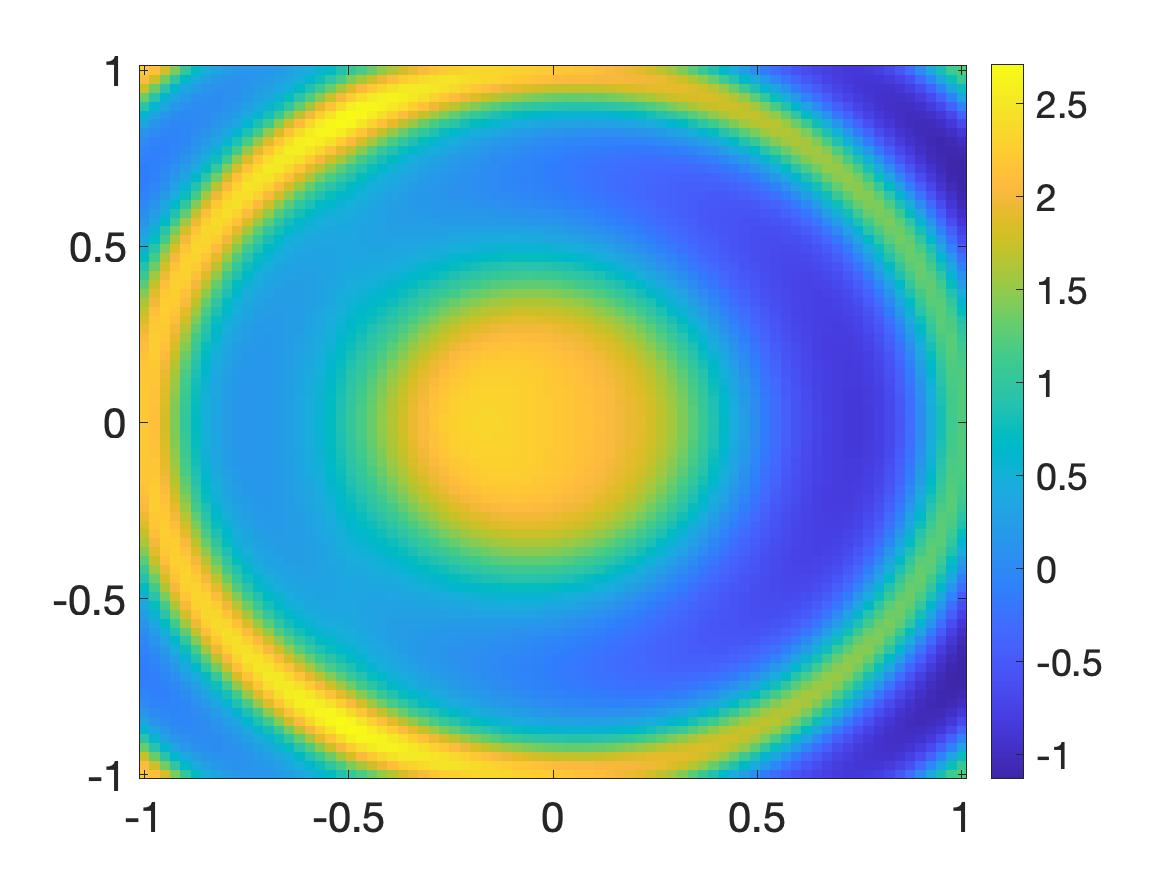}}
	\quad
	\subfloat[The initial solution $u_0$ computed by minimizing $J_0^{\lambda, \beta, \eta}$ defined in \eqref{J0}]{\includegraphics[width=.3\textwidth]{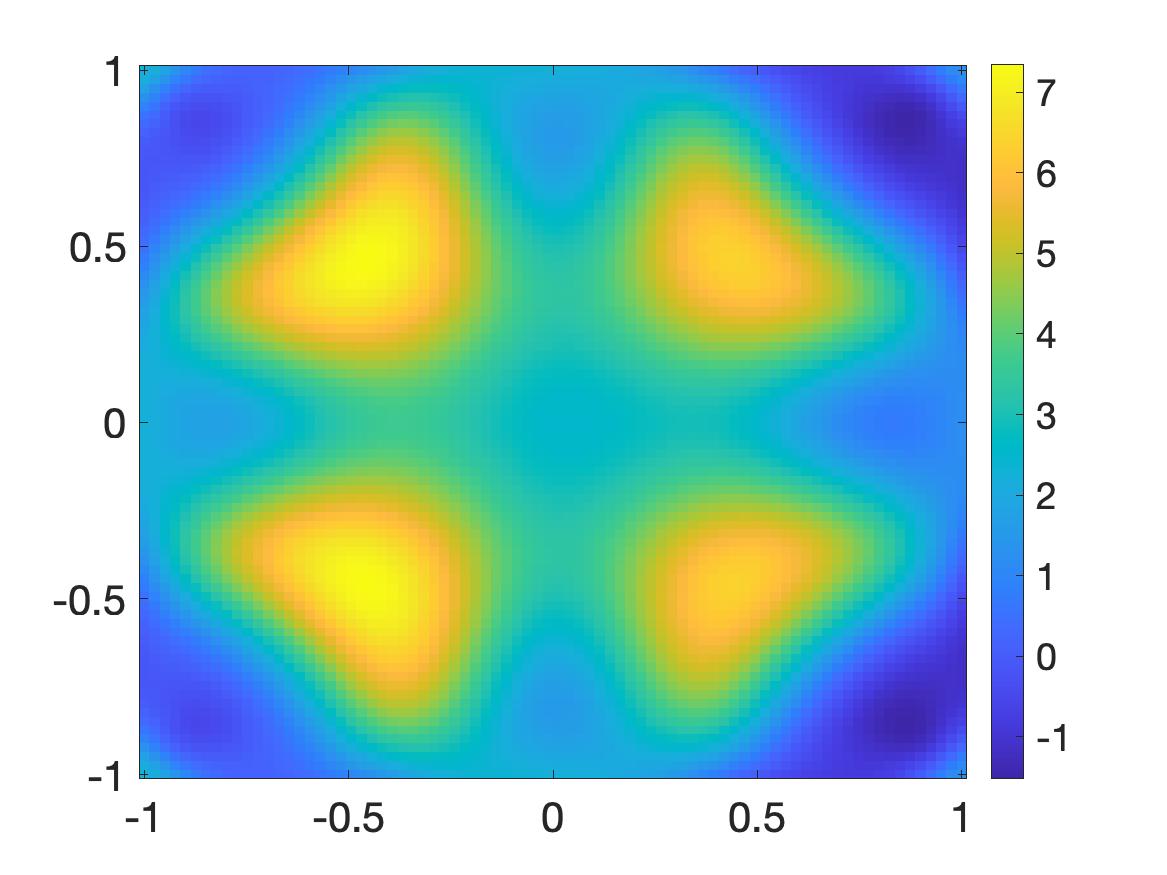}}
	\quad
	\subfloat[The computed solution $u_{\rm comp}$.]{\includegraphics[width=.3\textwidth]{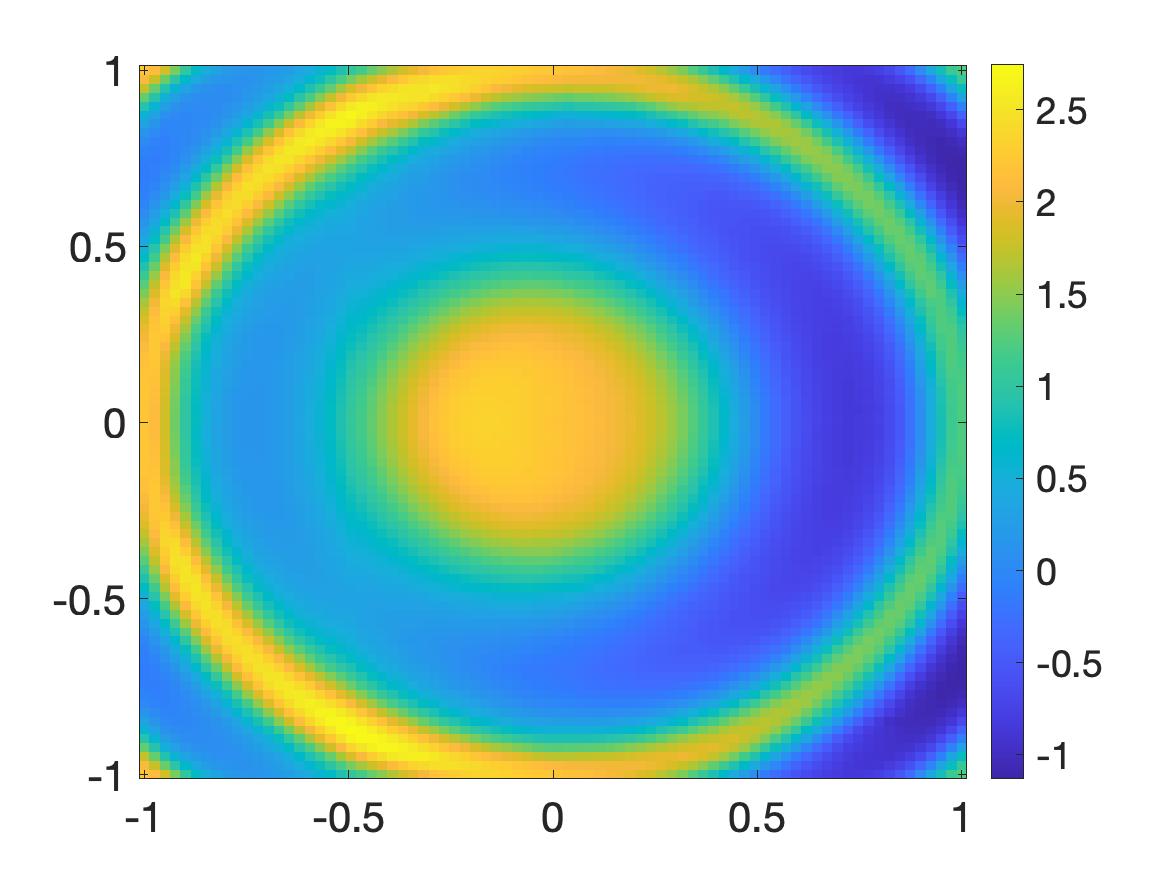}}

	\subfloat[$\|u_{n} - u_{n - 1}\|_{L^2(\Omega)}$. The horizontal axis is the number of iteration $n$.]{\includegraphics[width=.3\textwidth]{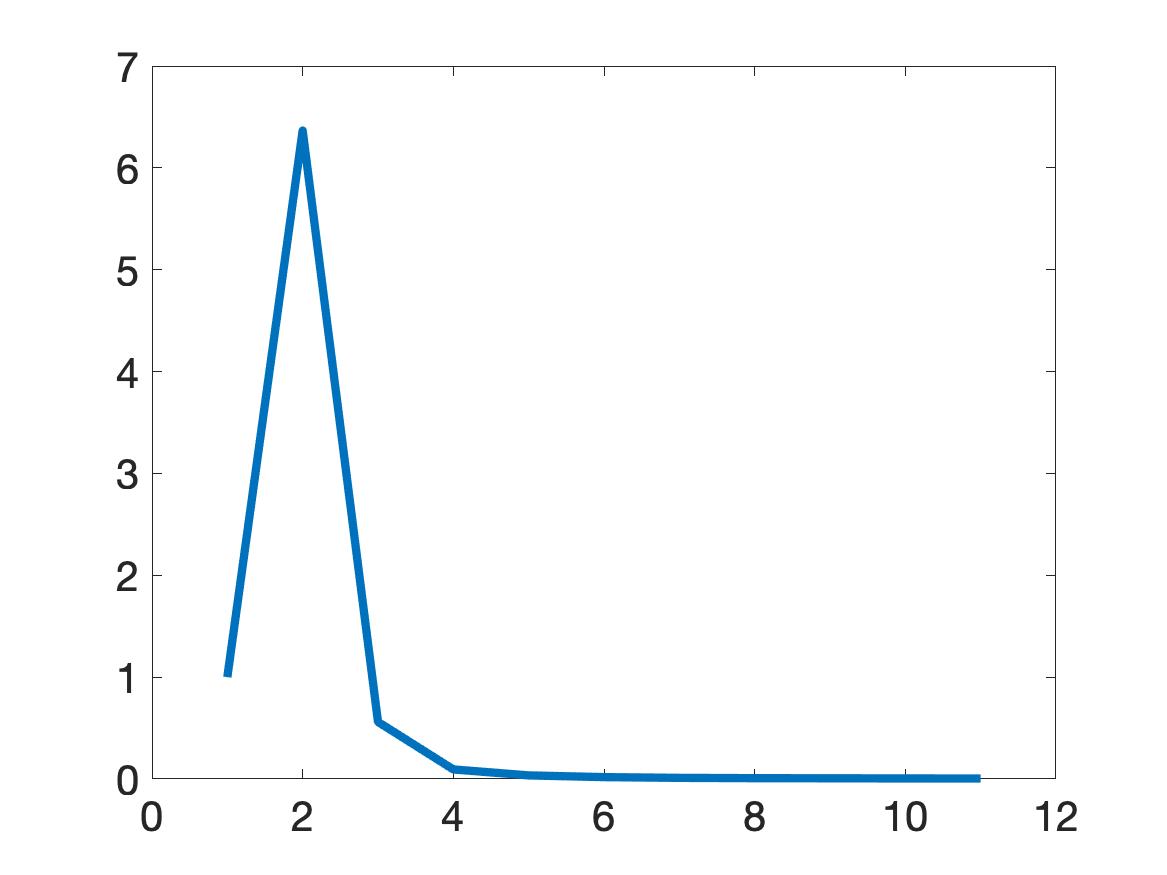}}
	\quad
	\subfloat[The true and computed solutions on the line $\{(x = 0 , y) \in \Omega\}$]{\includegraphics[width=.3\textwidth]{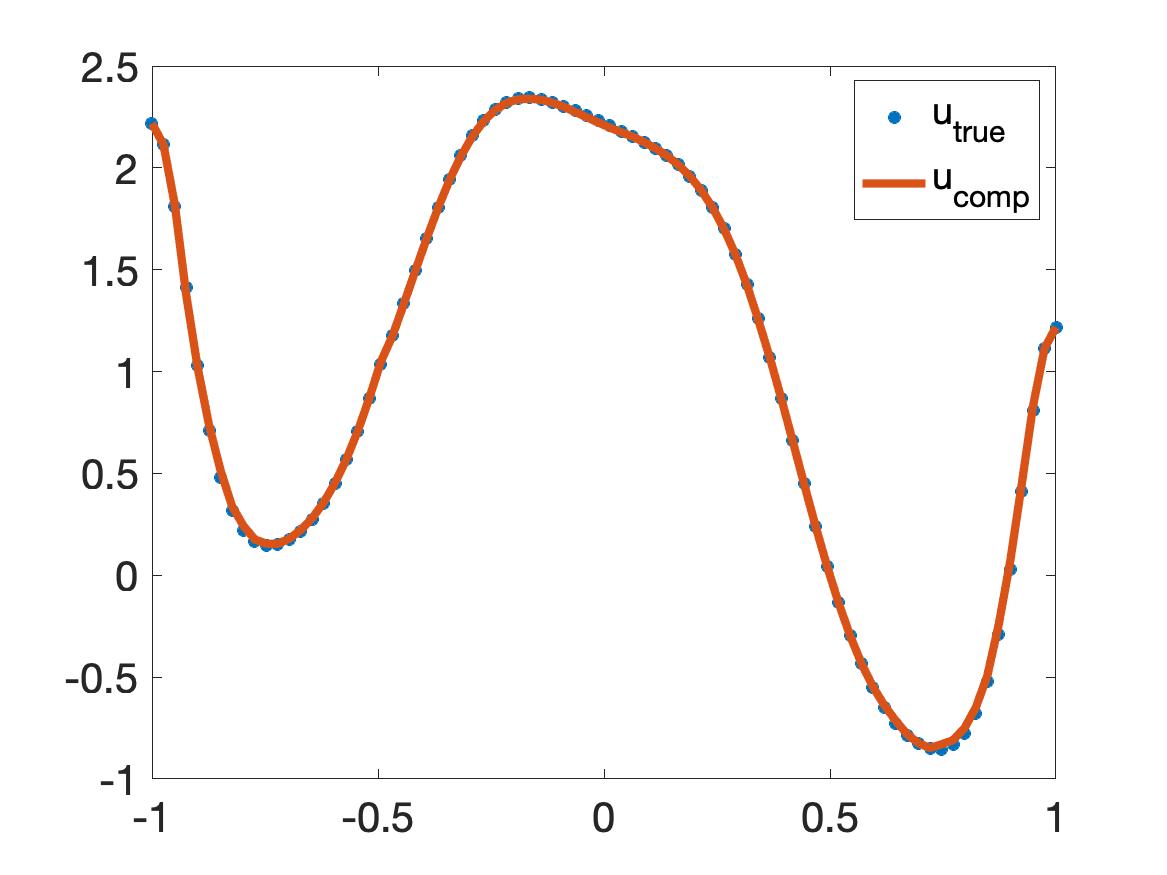}}
	\quad
	\quad
	\subfloat[\label{fig 3f} The relative error $\frac{|u^*(\x) - u_{\rm comp}(\x)|}{\|u_{\rm true}\|_{L^{\infty}(\Omega)}}.$]{\includegraphics[width = .3\textwidth]{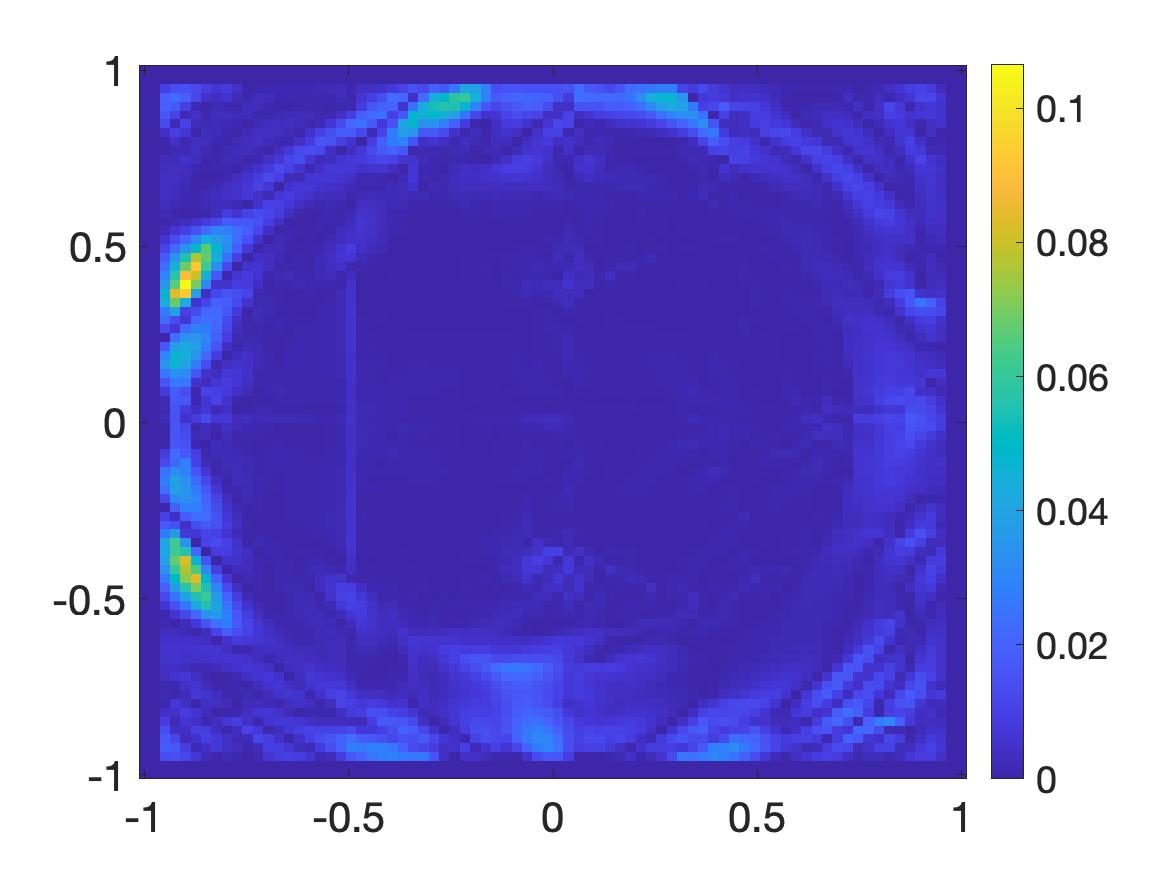}}
	
	\caption{\label{fig 3} Test 3. The true viscosity solution to \eqref{HJ}-\eqref{bdry} and the computed one. The Hamiltonian and the boundary conditions are given in \eqref{F3}-\eqref{boundary32}.}
\end{figure}
The relative error $ \frac{\|u^* - u_{\rm comp}\|_{L^\infty(\Omega)}}{\|u^*\|_{L^\infty(\Omega)}} = 10.65\%.$
Although the relative error in this test is large, the numerical result is acceptable. In fact, we observe from Figure \ref{fig 3f} that $ \frac{\|u^* - u_{\rm comp}\|_{L^\infty(\Omega)}}{\|u^*\|_{L^\infty(\Omega)}}$ is small almost everywhere  while it is large at only two small places near the left edge of the domain. 

 \medskip

{\it Test 4.} We test the case when the Hamiltonian is not increasing in the second variable and is not convex in the third variable. 
The Hamiltonian in this test is given by
\begin{multline}
	F(\x, s, \p) = -40s + ||\p| - 10|  
	+ 40\Big(|x + y - 0.5| + \sin\Big(\frac{x^2}{2} + y^2\Big)\Big) 
		\\-
		\Big|\Big(\Big[\mbox{sign}(x + y -  0.5) + x\cos\Big(\frac{x^2}{2} + y^2\Big)\Big]^2 +\\
		 \Big[
		\mbox{sign}(x + y -  0.5) + 2y\cos\Big(\frac{x^2}{2} + y^2\Big)
		\Big]^2\Big)^{1/2}
		- 10\Big|
	\label{F4}
\end{multline}
for all $\x = (x, y) \in \Omega, s \in \R, \p = (p_1, p_2) \in \R^2$.
The boundary conditions are given by
\begin{equation}
	u(\x) = |x + y - 0.5| + \sin\Big(\frac{x^2}{2} + y^2\Big)
	\quad \mbox{for all } \x = (x, y) \in \partial \Omega
	\label{boundary41}
\end{equation}
and
\begin{multline}
	 \partial_\nu u(\x) 	 = \Big(\mbox{sign}(x + y -  0.5) + x\cos\Big(\frac{x^2}{2} + y^2\Big), \\
	 \mbox{sign}(x + y -  0.5) + 2y\cos\Big(\frac{x^2}{2} + y^2\Big)\Big) \cdot \nu
	 \label{boundary42}
\end{multline}
 for all $\x = (x, y) \in \partial \Omega$.
 The true solution is $u^*(\x) = |x + y - 0.5| + \sin\Big(\frac{x^2}{2} + y^2\Big)$.
The graphs of $u^*$ and $u_{\rm comp}$ are displayed in Figure \ref{fig 4}.
 \begin{figure}[h!]
	\subfloat[The true solution $u^*$ to the Hamilton-Jacobi equation where the Hamiltonian is given in \eqref{F4}.]{\includegraphics[width=.3\textwidth]{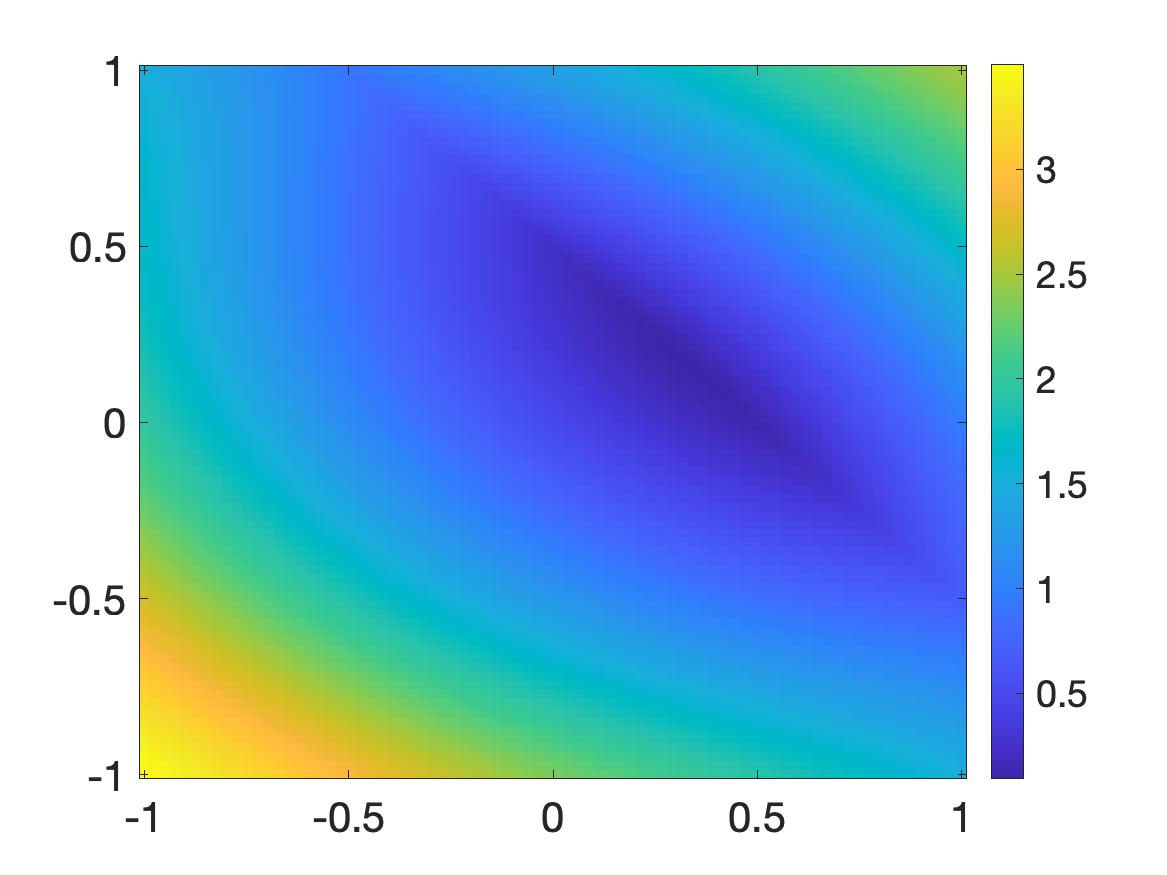}}
	\quad
	\subfloat[The initial solution $u_0$ computed by minimizing $J_0^{\lambda, \beta, \eta}$ defined in \eqref{J0}]{\includegraphics[width=.3\textwidth]{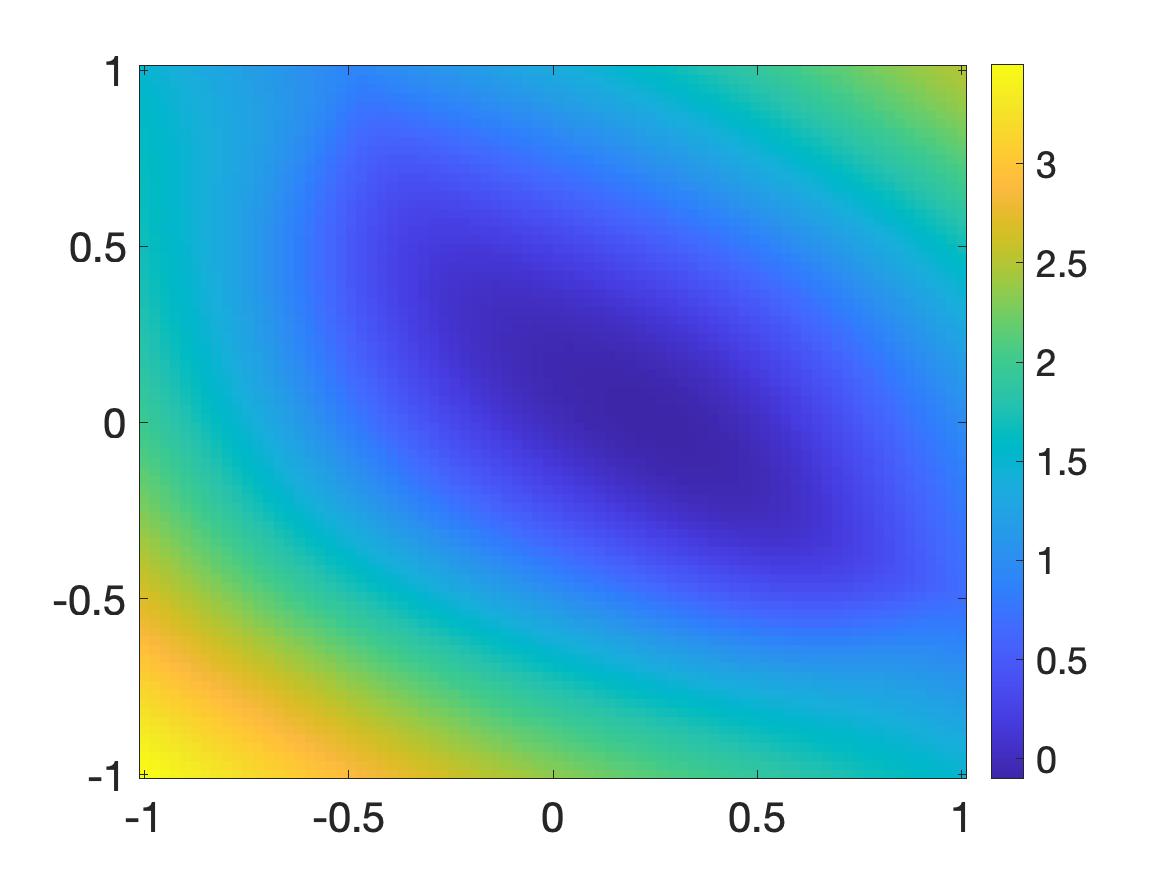}}
	\quad
	\subfloat[The computed solution $u_{\rm comp}$.]{\includegraphics[width=.3\textwidth]{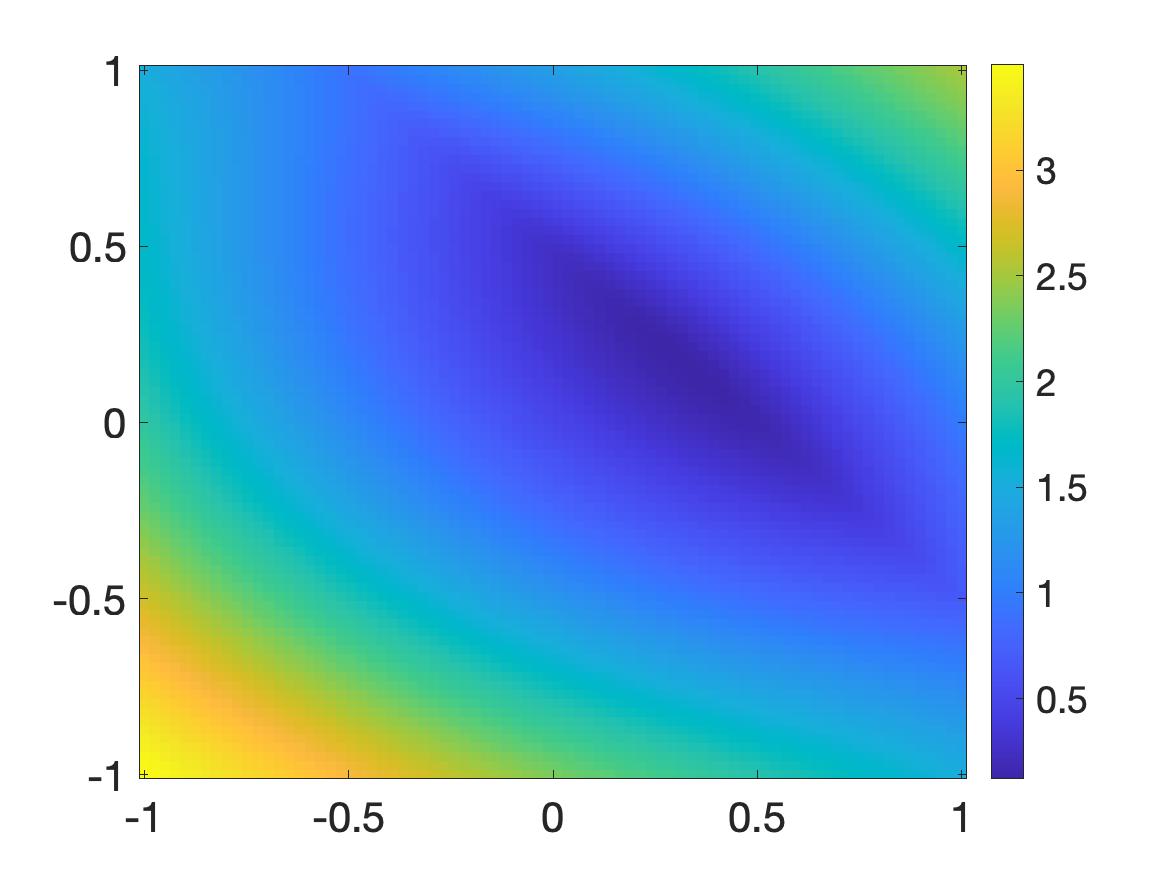}}

	\subfloat[The relative error $\|u_{n} - u_{n - 1}\|_{L^2(\Omega)}$. The horizontal axis is the number of iteration $n$.]{\includegraphics[width=.3\textwidth]{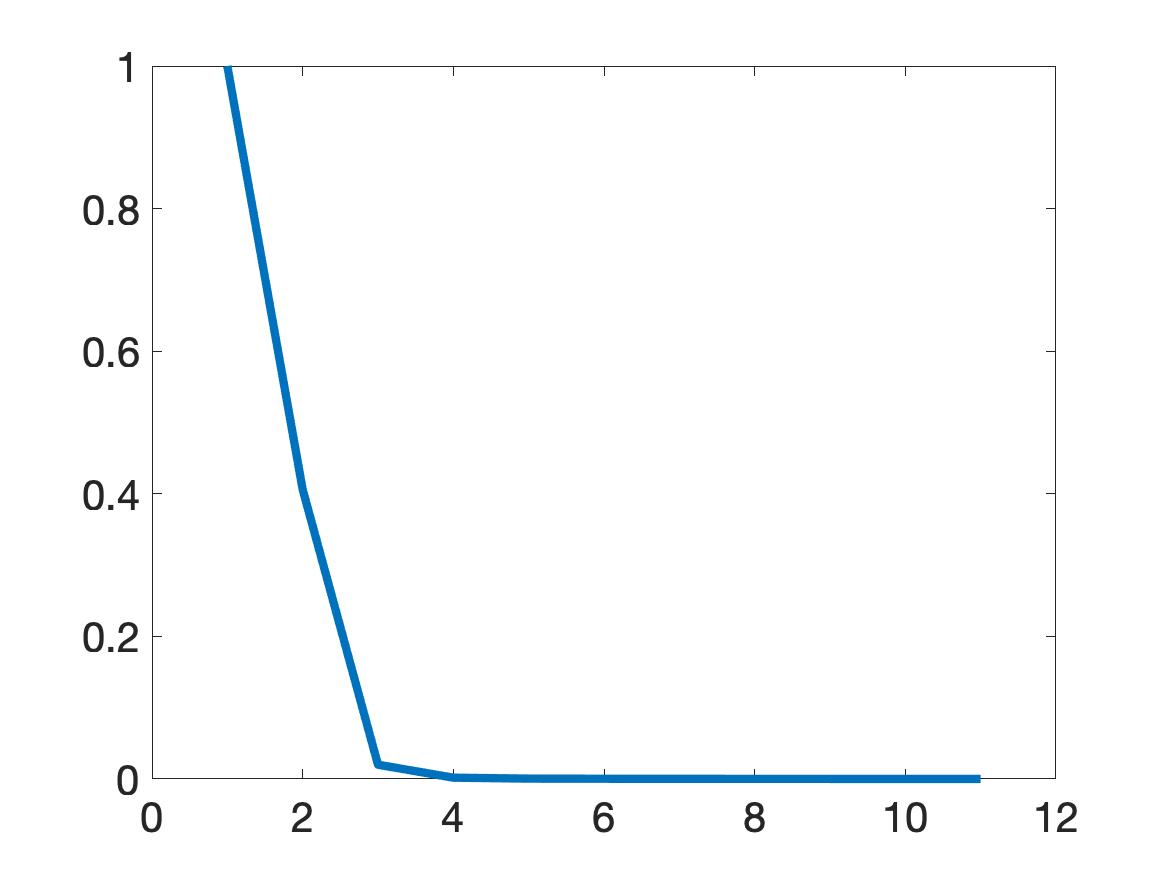}}
	\quad
	\subfloat[The true and computed solutions on the line $\{(x = 0 , y) \in \Omega\}$]{\includegraphics[width=.3\textwidth]{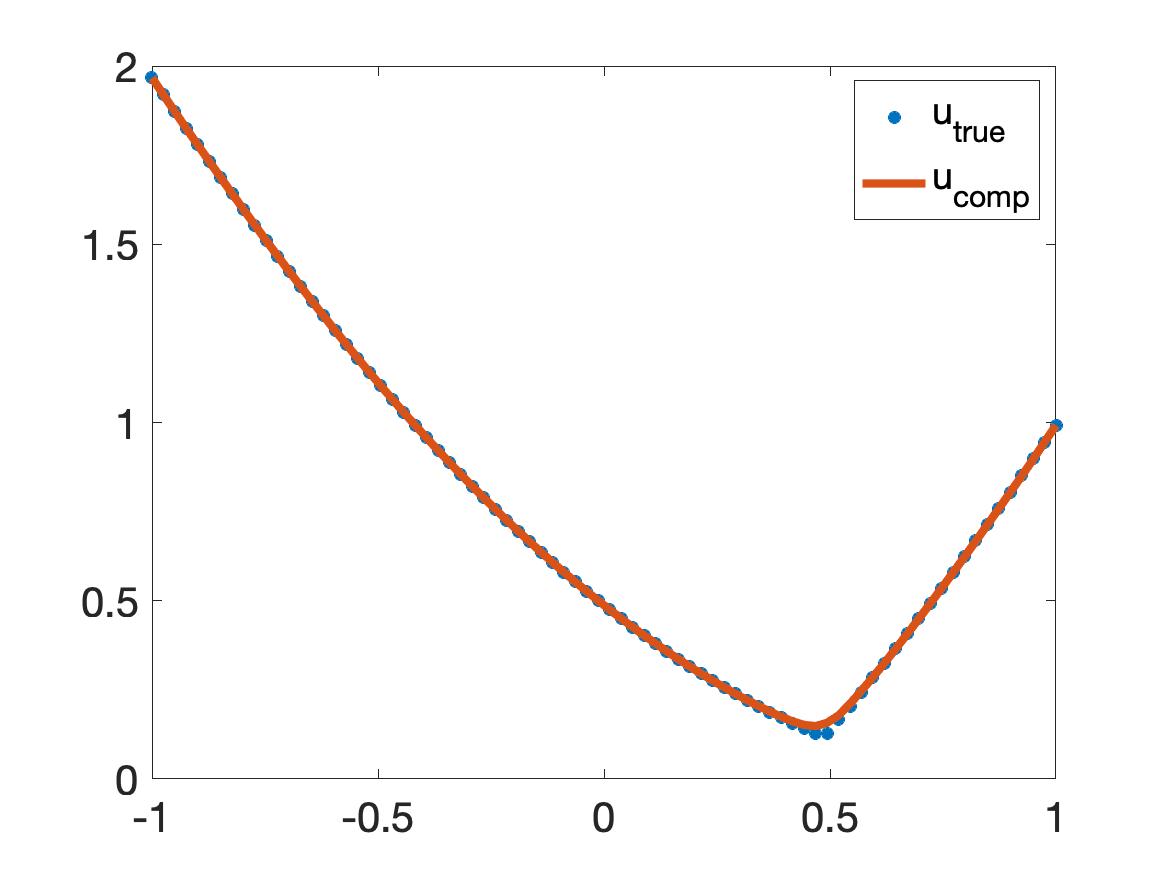}}
	\quad
	\subfloat[ The relative error $\frac{|u^*(\x) - u_{\rm comp}(\x)|}{\|u_{\rm true}\|_{L^{\infty}(\Omega)}}.$]{\includegraphics[width = .3\textwidth]{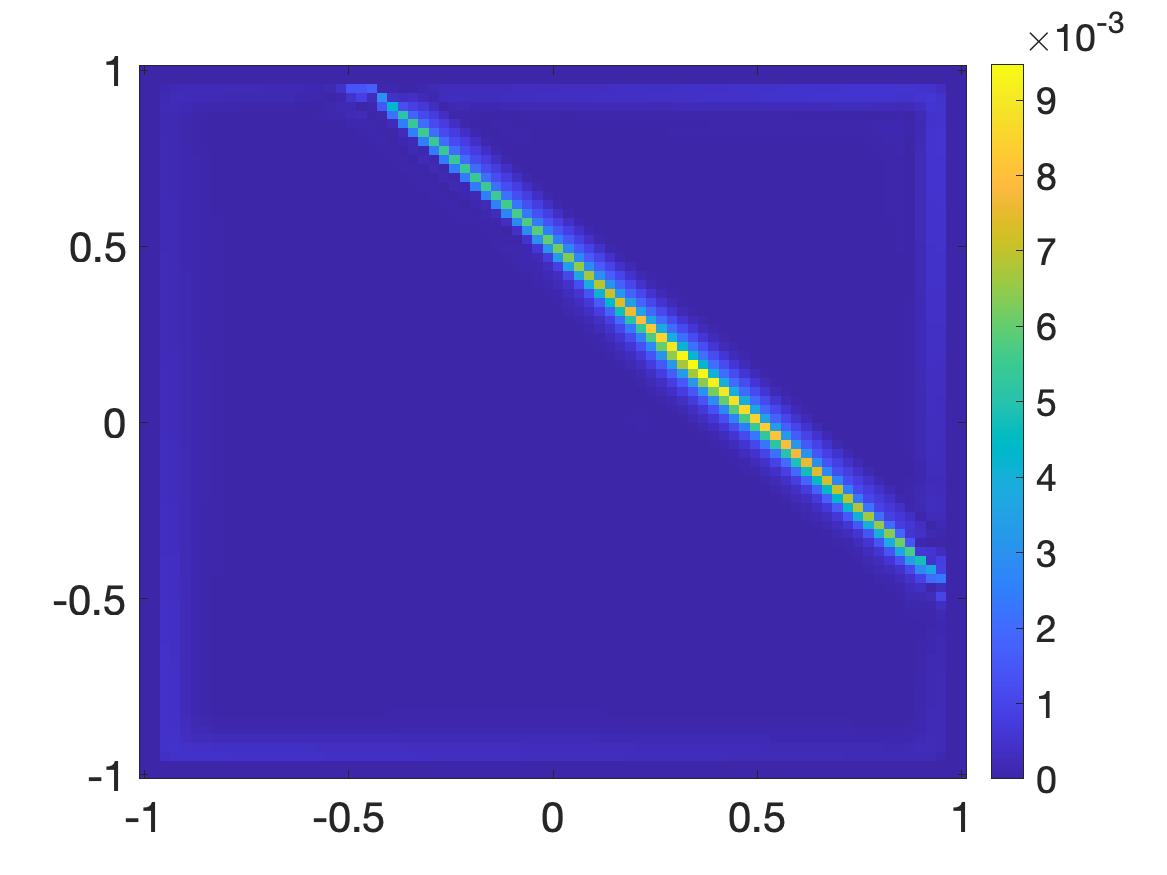}}
	\caption{\label{fig 4} Test 4.  The true viscosity solution to \eqref{HJ}-\eqref{bdry} and the computed one. The Hamiltonian and the boundary conditions are given in \eqref{F4}-\eqref{boundary42}.}
\end{figure}
The relative error $\frac{\|u^* - u_{\rm comp}\|_{L^\infty(\Omega)}}{\|u^*\|_{L^\infty(\Omega)}} = 0.95\%.$
 
 \medskip

{\it Test 5.} We solve a G-equation. 
The Hamiltonian in this test is given by
\begin{equation}
	F(\x, s, \p) = 5 s + |\p| - x p_1 
	+
	\Big[
		5(|x - 0.5| + |y|) -x \mbox{sign}(x - 0.5) - \sqrt{2} 
	\Big]
	\label{F5}
\end{equation}
for all $\x = (x, y) \in \Omega, s \in \R, \p = (p_1, p_2) \in \R^2$.
The boundary conditions are given by
\begin{equation}
	u(\x) = -|x - 0.5| - |y|
	\quad \mbox{for all } \x = (x, y) \in \partial \Omega
	\label{boundary51}
\end{equation}
and
\begin{equation}
	 \partial_\nu u(\x) = -\big(\mbox{sign}(x - 0.5), \mbox{sign}(y) 
	 \big) \cdot \nu
	 \label{boundary52}
\end{equation}
 for all $\x = (x, y) \in \partial \Omega$.
 The true solution is $u^*(\x) =- |x - 0.5| - |y|$.
The graphs of $u^*$ and $u_{\rm comp}$ are displayed in Figure \ref{fig 3}.
 \begin{figure}[h!]
	\subfloat[The true solution $u^*$ to the Hamilton-Jacobi equation where the Hamiltonian is given in \eqref{F5}.]{\includegraphics[width=.3\textwidth]{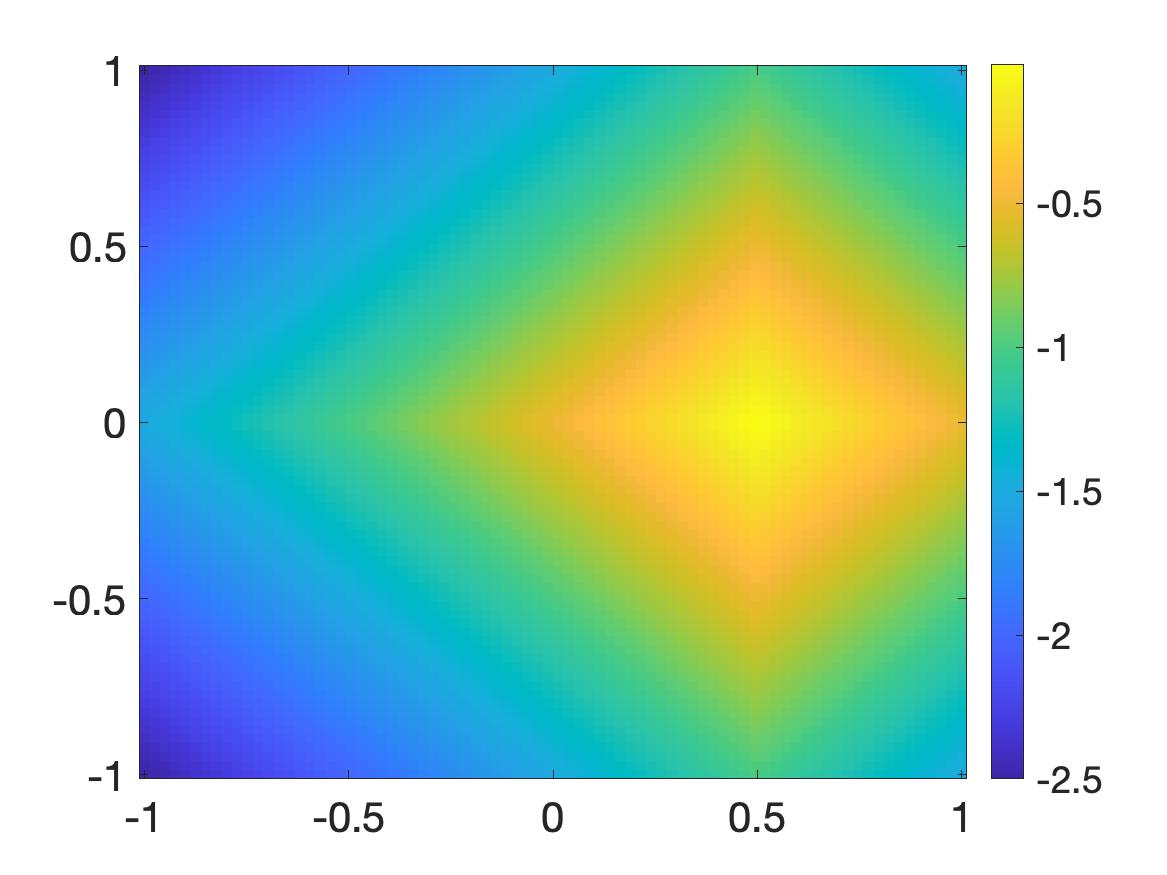}}
	\quad
	\subfloat[The initial solution $u_0$ computed by minimizing $J_0^{\lambda, \beta, \eta}$ defined in \eqref{J0}]{\includegraphics[width=.3\textwidth]{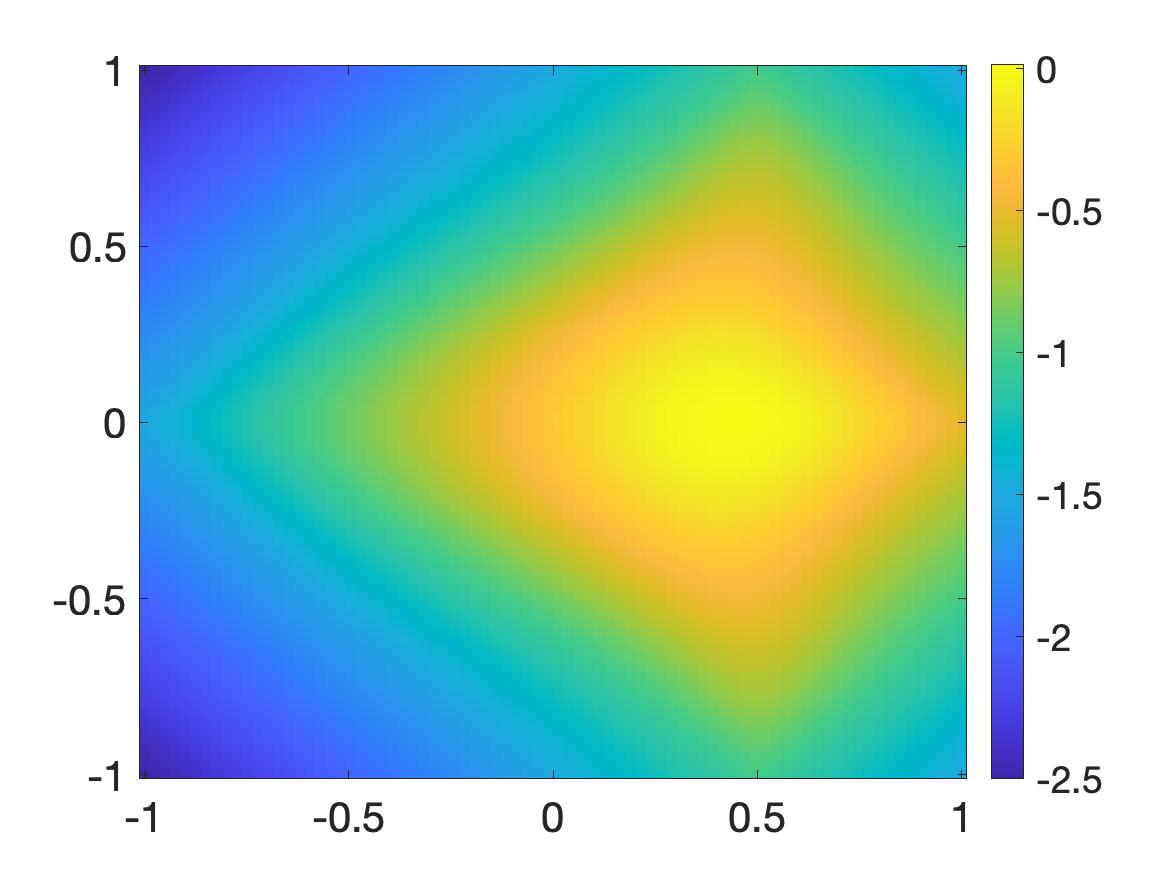}}
	\quad
	\subfloat[The computed solution $u_{\rm comp}$.]{\includegraphics[width=.3\textwidth]{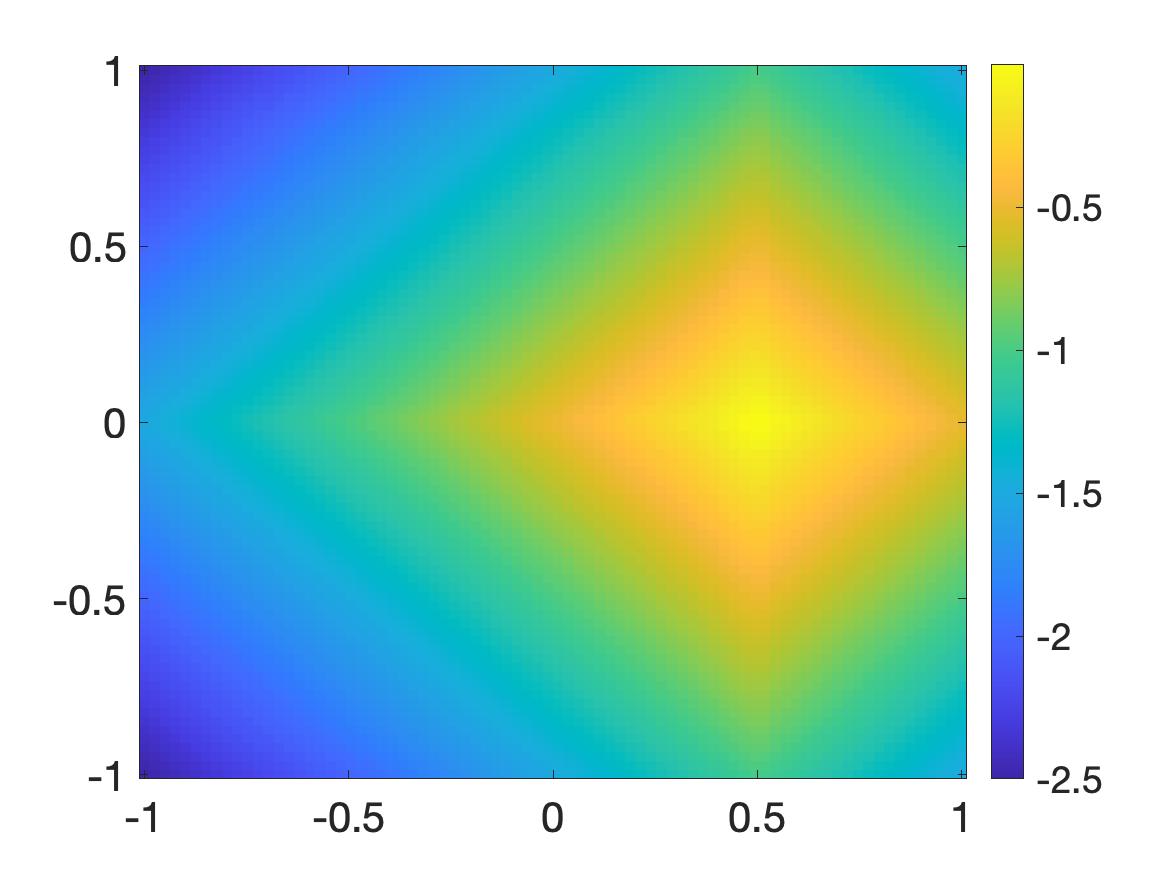}}

	\subfloat[The relative error $\|u_{n} - u_{n - 1}\|_{L^2(\Omega)}$. The horizontal axis is the number of iteration $n$.]{\includegraphics[width=.3\textwidth]{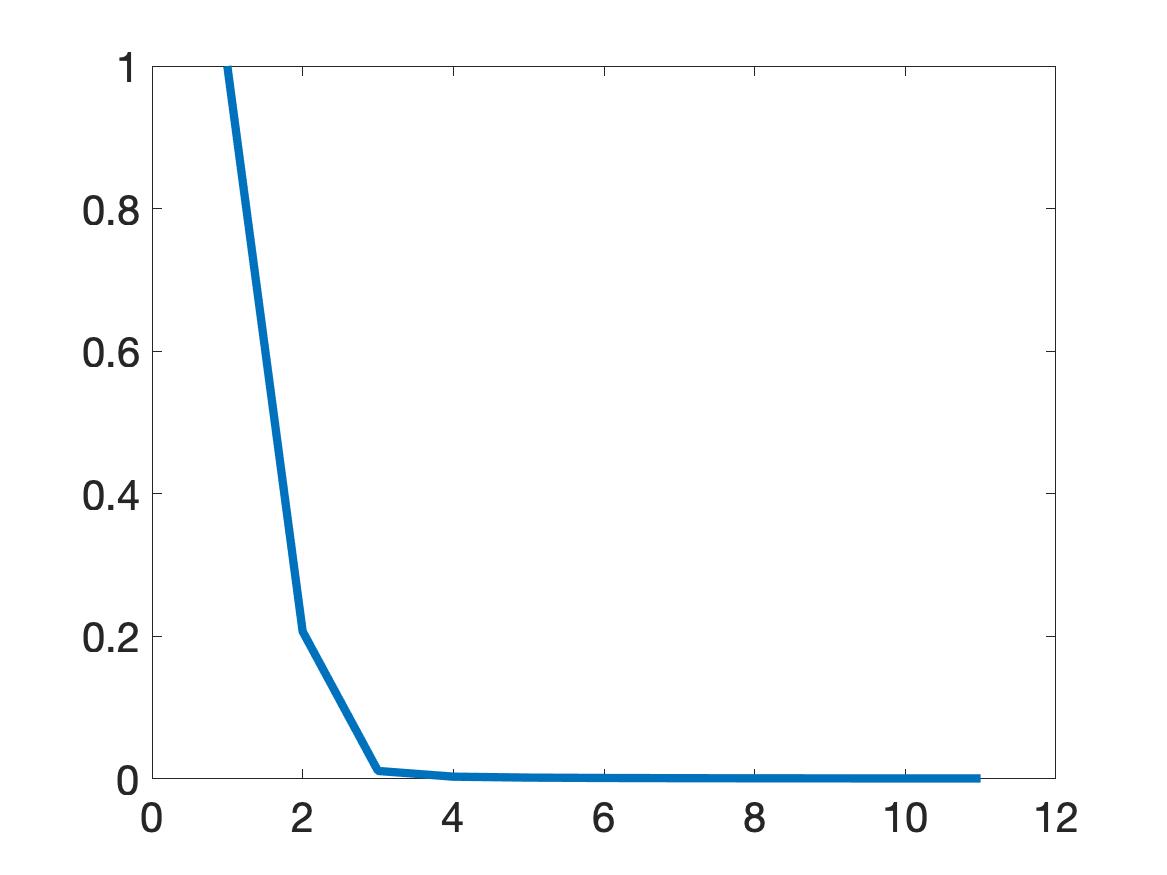}}
	\quad
	\subfloat[The true and computed solutions on the line $\{(x = 0 , y) \in \Omega\}$]{\includegraphics[width=.3\textwidth]{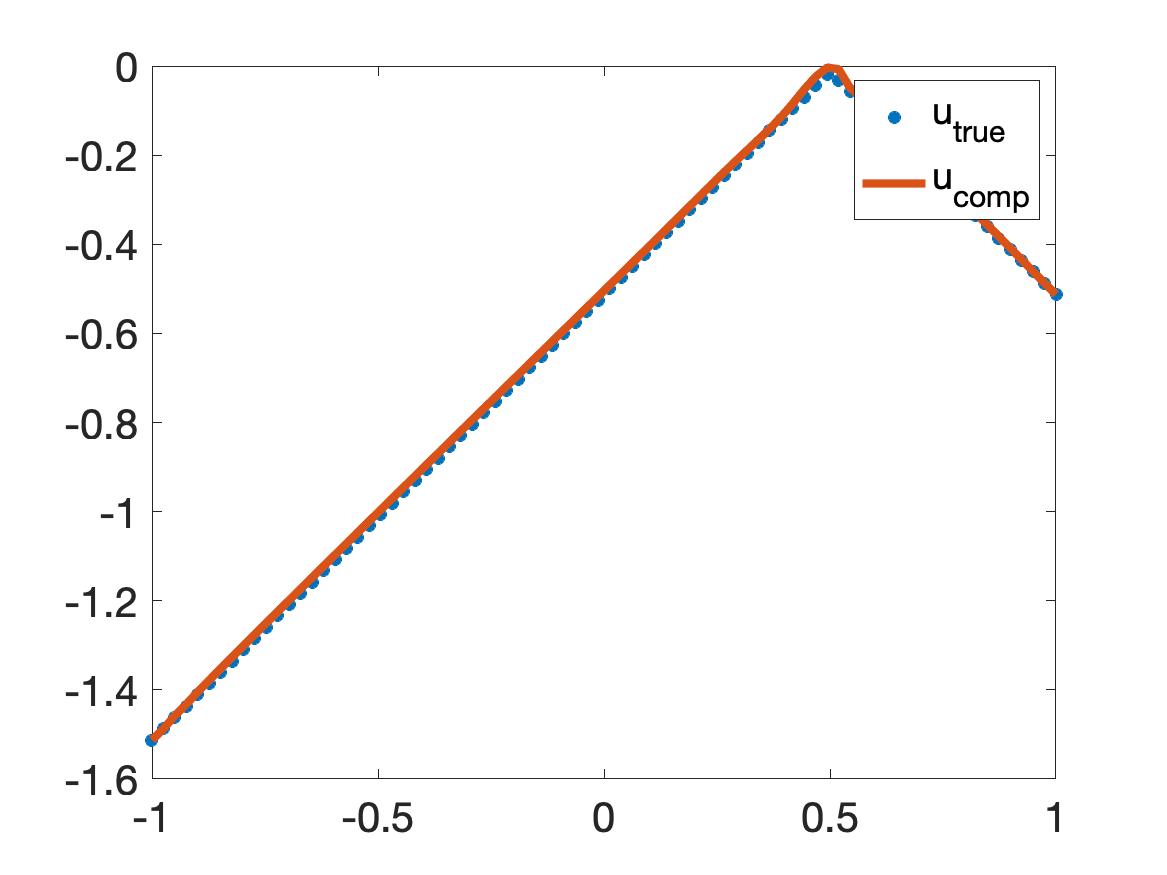}}
	\quad
	\subfloat[ The relative error $\frac{|u^*(\x) - u_{\rm comp}(\x)|}{\|u_{\rm true}\|_{L^{\infty}(\Omega)}}.$]{\includegraphics[width = .3\textwidth]{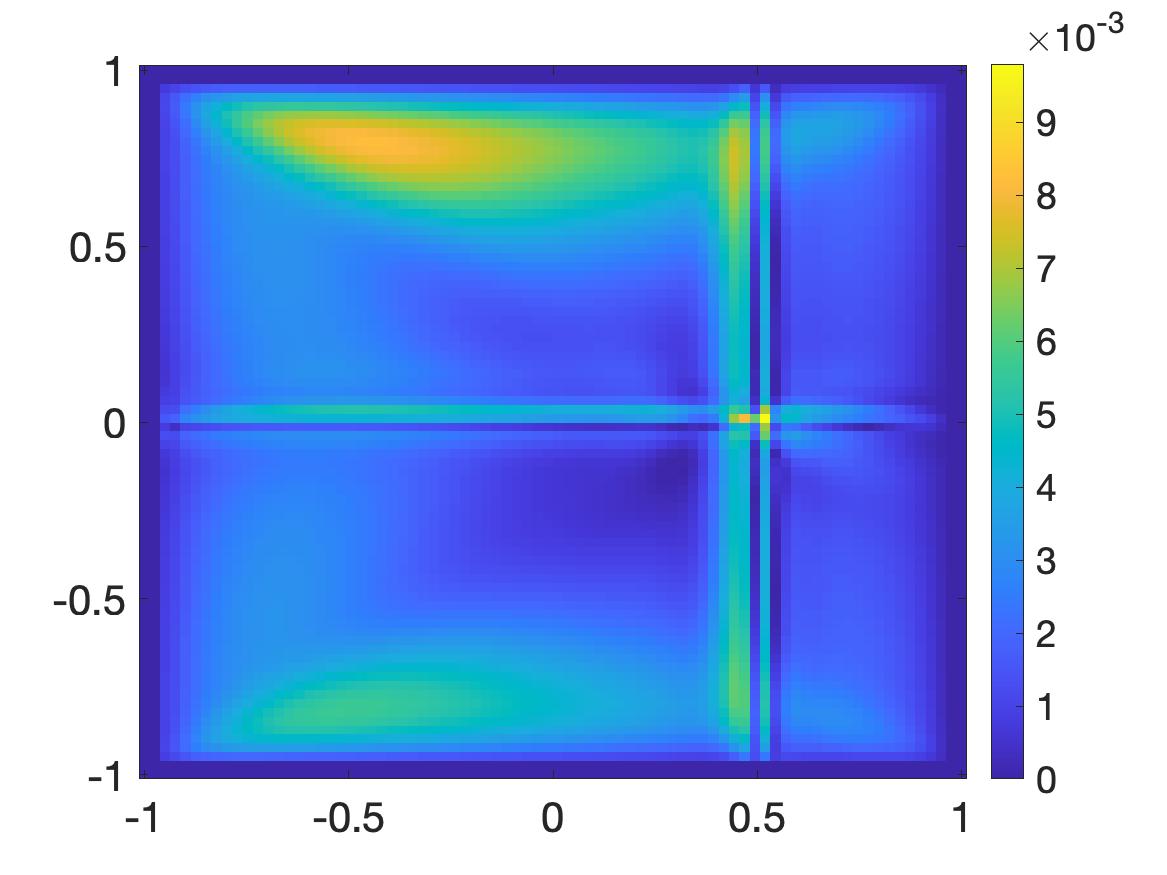}}
	\caption{\label{fig 5} Test 5. The true viscosity solution to \eqref{HJ}-\eqref{bdry} and the computed one. The Hamiltonian and the boundary conditions are given in \eqref{F5}-\eqref{boundary52}.}
\end{figure}
The relative error $\frac{\|u^* - u_{\rm comp}\|_{L^\infty(\Omega)}}{\|u^*\|_{L^\infty(\Omega)}} = 0.98\%.$

  \medskip

{\it Test 6.}  The Hamiltonian in this test is given by
\begin{multline}
	F(\x, s, \p) = 20 s + \min\{|\p|, ||\p| - 10| + 6\} 
	-
	\Big[
		20 (-|x| + \sin(\pi(x^2 + y^2))) 
		\\
		+ \min\Big\{
			h(x, y), |h(x, y) - 10| + 6
		\Big\}
	\Big]
	\label{F6}
\end{multline}
for all $\x = (x, y) \in \Omega, s \in \R, \p = (p_1, p_2) \in \R^2$
where
\[
	h(x, y) = \sqrt{[-\mbox{sign}(x) + 2\pi \cos(\pi(x^2 + y^2))]^2 + [2\pi \cos(\pi(x^2 + y^2)]^2}.
\]
The boundary conditions are given by
\begin{equation}
	u(\x) = -|x| + \sin(\pi(x^2 + y^2))
	\quad \mbox{for all } \x = (x, y) \in \partial \Omega
		\label{boundary61}
\end{equation}
and
\begin{equation}
	 \partial_\nu u(\x) = \big(
	 	-\mbox{sign}(x) + 2\pi x \cos(\pi(x^2 + y^2)), 2\pi y \cos(\pi(x^2 + y^2))
	 \big) \cdot \nu
	 \label{boundary62}
\end{equation}
 for all $\x = (x, y) \in \partial \Omega$.
 The true solution is $u^*(\x) =  -|x| + \sin(\pi(x^2 + y^2))$.
The graphs of $u^*$ and $u_{\rm comp}$ are displayed in Figure \ref{fig 6}.
 \begin{figure}[h!]
	\subfloat[The true solution $u^*$ to the Hamilton-Jacobi equation where the Hamiltonian is given in \eqref{F6}.]{\includegraphics[width=.3\textwidth]{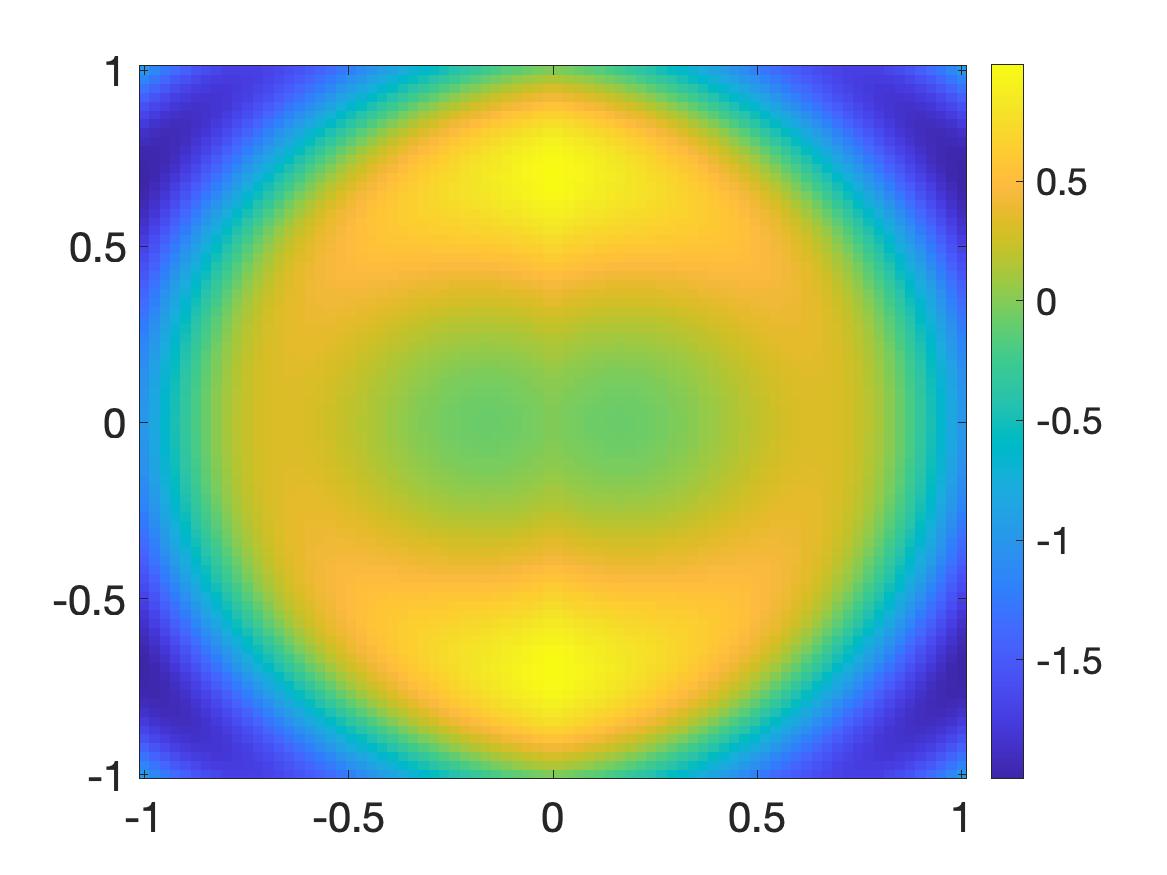}}
	\quad
	\subfloat[The initial solution $u_0$ computed by minimizing $J_0^{\lambda, \beta, \eta}$ defined in \eqref{J0}]{\includegraphics[width=.3\textwidth]{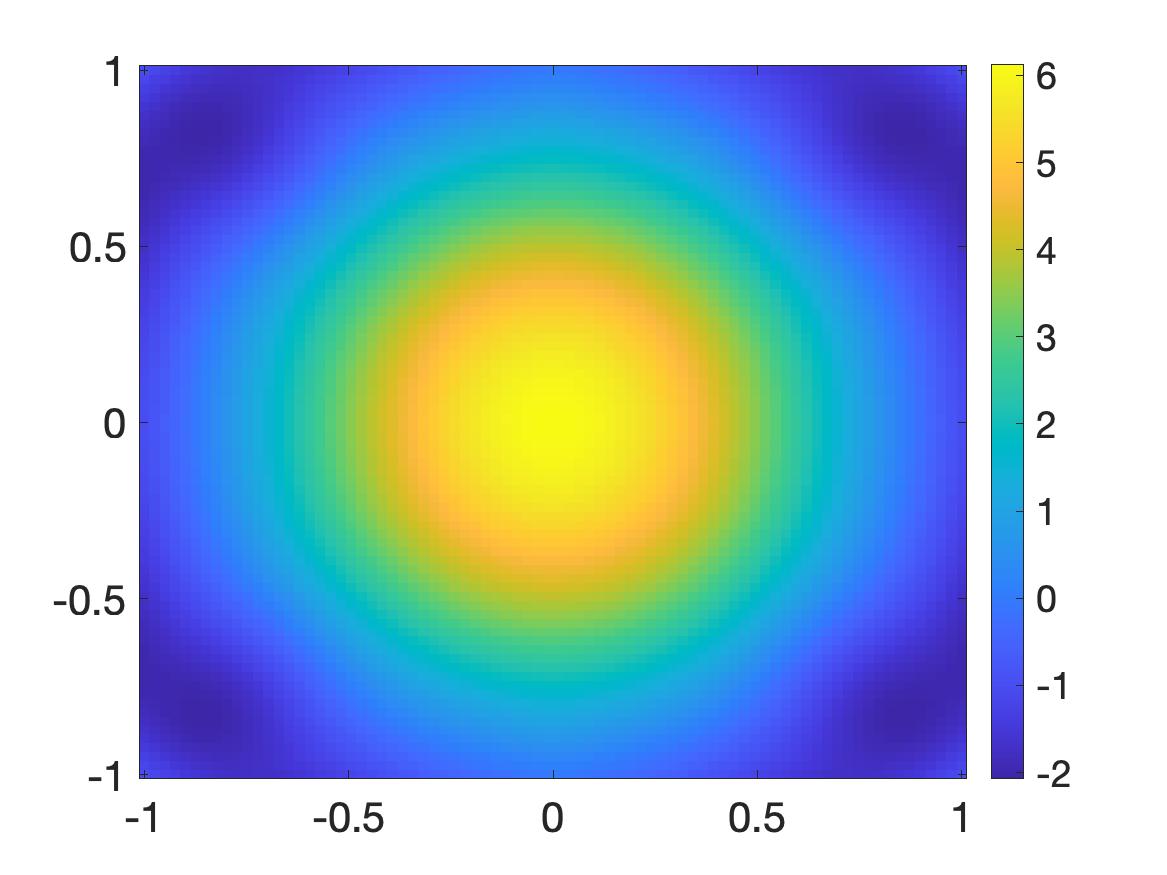}}
	\quad
	\subfloat[The computed solution $u_{\rm comp}$.]{\includegraphics[width=.3\textwidth]{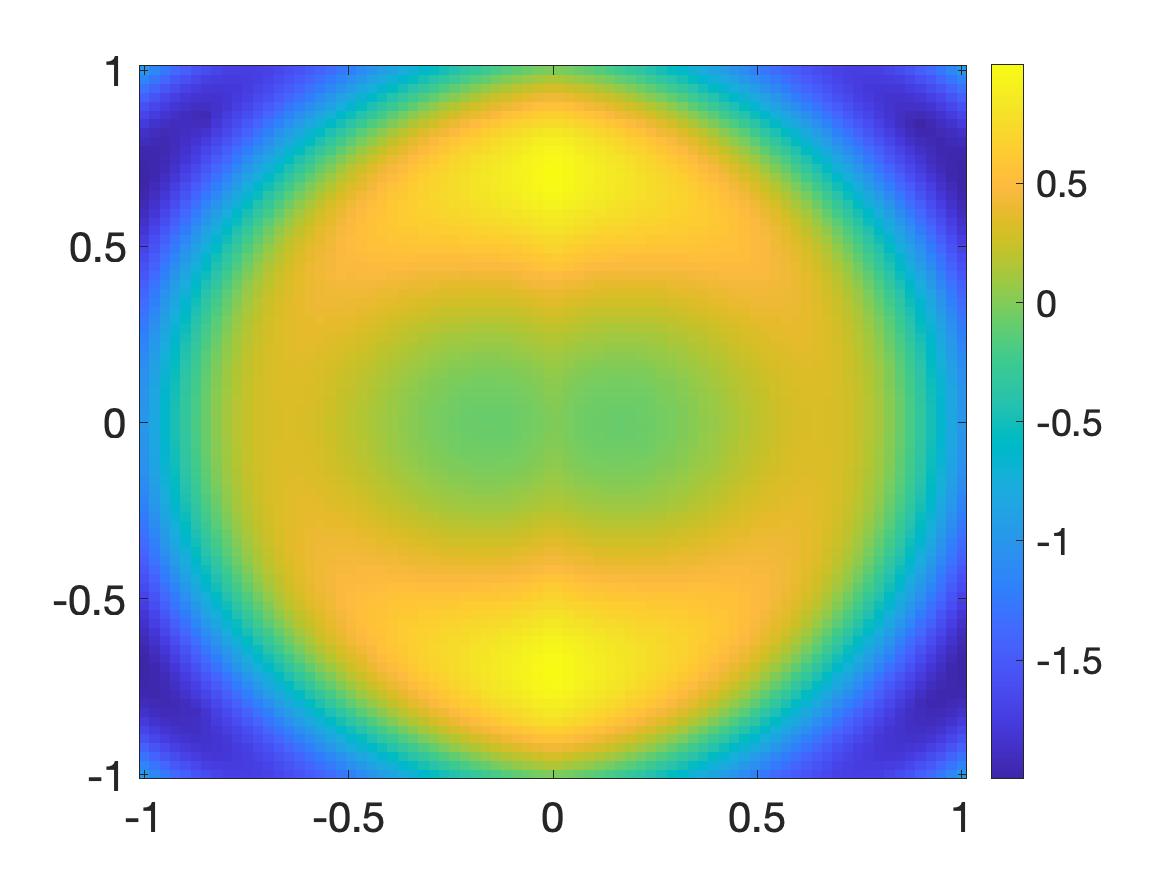}}

	\subfloat[The relative error $\|u_{n} - u_{n - 1}\|_{L^2(\Omega)}$. The horizontal axis is the number of iteration $n$.]{\includegraphics[width=.3\textwidth]{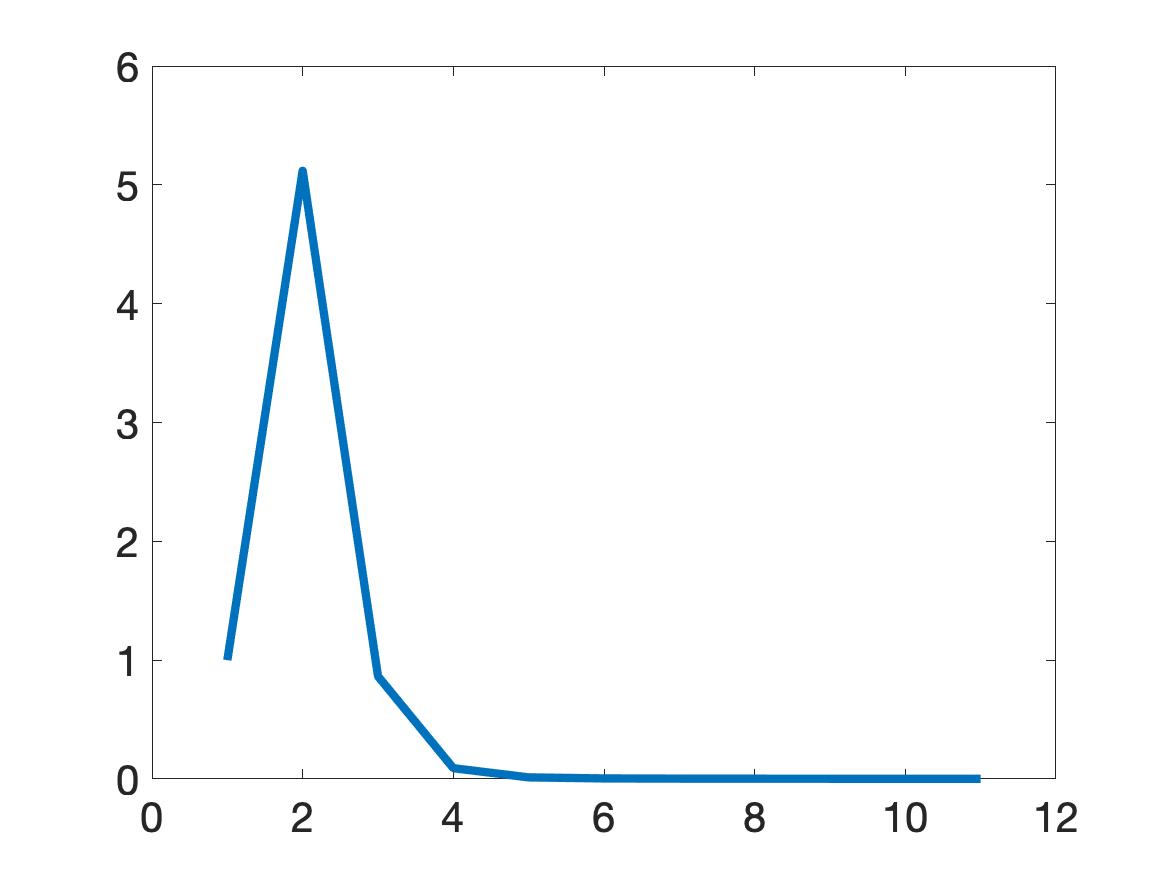}}
	\quad
	\subfloat[The true and computed solutions on the line $\{(x = 0 , y) \in \Omega\}$]{\includegraphics[width=.3\textwidth]{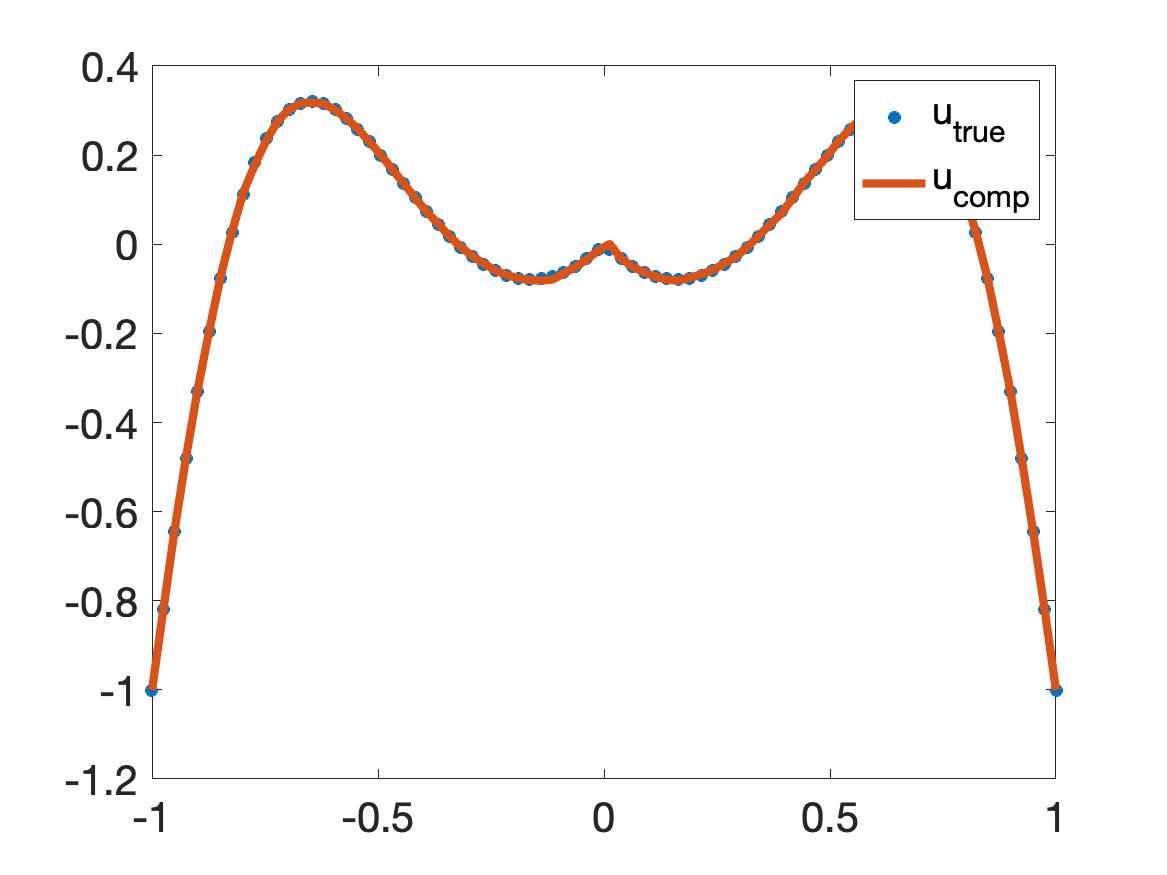}}
	\quad
	\subfloat[ The relative error $\frac{|u^*(\x) - u_{\rm comp}(\x)|}{\|u_{\rm true}\|_{L^{\infty}(\Omega)}}.$]{\includegraphics[width = .3\textwidth]{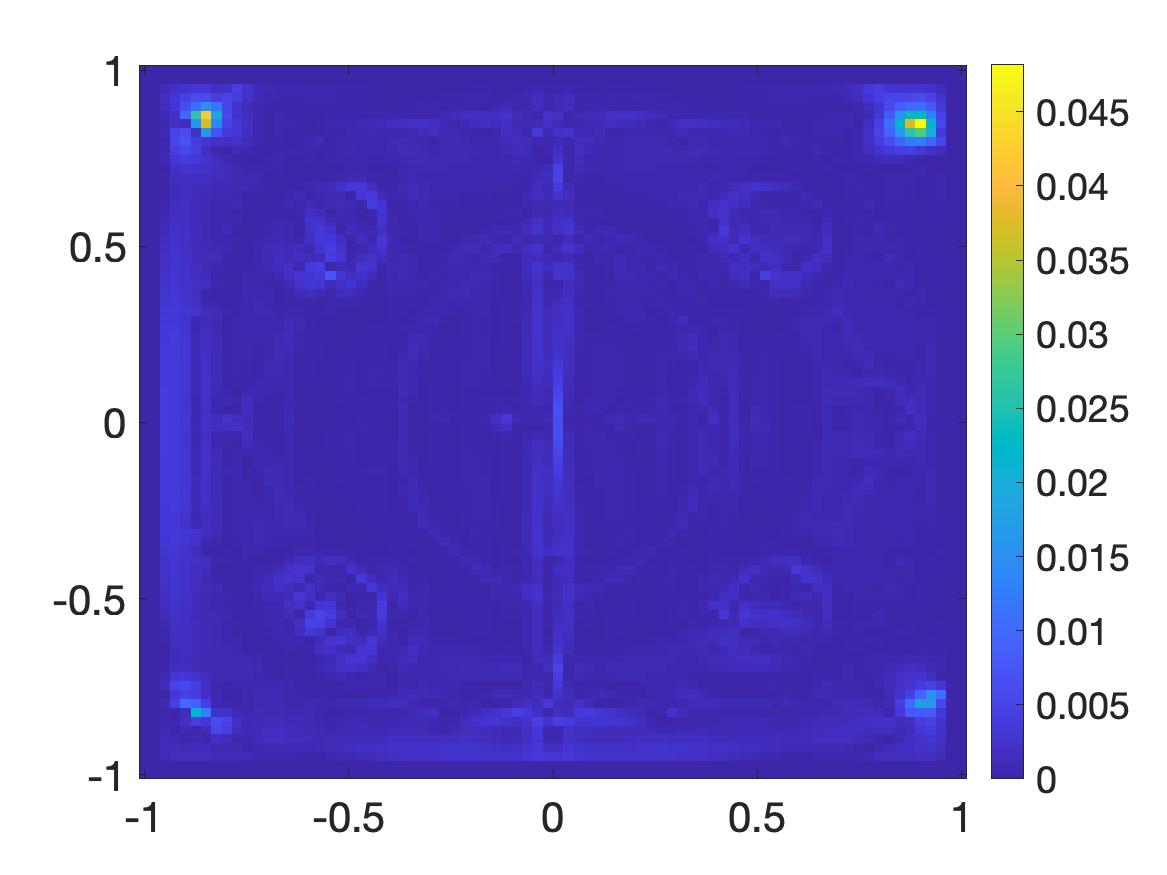}}
	\caption{\label{fig 6} Test 6. The true viscosity solution to \eqref{HJ}-\eqref{bdry} and the computed one. The Hamiltonian and the boundary conditions are given in \eqref{F6}-\eqref{boundary62}.}
\end{figure}
The relative error $\frac{\|u^* - u_{\rm comp}\|_{L^\infty(\Omega)}}{\|u^*\|_{L^\infty(\Omega)}} = 4.8\%.$

 \begin{remark}
 	The $L^\infty$ relative errors in all tests above for first-order Hamilton-Jacobi equations are compatible with $\max\{O(\sqrt{\eta}),O(\sqrt{\epsilon_0})\} \simeq 3\%.$
 \end{remark}

\section{Concluding remarks}  \label{sec 5}
 
We have developed a new globally convergent numerical method to solve over-determined boundary value problems of quasilinear elliptic equations. 
The key point of the method is to repeatedly solve the linearization of the given PDE by the Carleman weighted quasi-reversibility method (Algorithm \ref{alg 1}).  
As the result, we obtain a sequence of functions converging to the solution thanks to the Carleman estimate (Lemma \ref{lem:Carleman}).
The strength of our new method includes (1) the global convergence property and (2) the fast convergence rate, which is described in Theorem \ref{thm:main}.
Some numerical results for quasilinear elliptic equations and first-order Hamilton-Jacobi equations are presented to show the effectiveness of our method.

\section*{Acknowledgement} The works of TTL and LHN were supported by US Army Research Laboratory and US Army Research
Office grant W911NF-19-1-0044 and , in part, by funds provided by the Faculty Research Grant program at UNC Charlotte, Fund No. 111272. 
The work of HT is supported in part by  NSF CAREER grant DMS-1843320 and a Simons fellowship.

\bibliographystyle{plain}
\bibliography{../../../../mybib}

\end{document}